\documentclass[a4paper,USenglish,numberwithinsect,thm-restate]{lipics-v2021}
\pdfoutput=1 
\hideLIPIcs

\usepackage{mathtools,xspace} 
\usepackage{tikz}
\usetikzlibrary{decorations.pathreplacing}
\usepackage{algpseudocode}
\usepackage{algorithm}
\usepackage{colorblind}

\usetikzlibrary{positioning}
\usetikzlibrary{arrows}
\usetikzlibrary{decorations.pathmorphing}
\usetikzlibrary{decorations.markings}
\usetikzlibrary{shapes.geometric}
\tikzset{
    small circles/.style={circle,inner sep=2pt,fill=#1},
    hollow circles/.style n args={2}{circle,inner sep=#1,draw=#2,thick},
    stars/.style={star,inner sep=2pt}
}

\newtheorem{question}{Question}
\newcommand{\majcol}[1]{\textbf{[#1]}}
\newcommand{\configuration}[1]{$\textnormal{config}_{#1}$}

\newcommand{\floor}[1]{\lfloor #1 \rfloor}

\tikzset{
  plain/.style={fill=none,shape=rectangle,draw=none,inner sep=3pt}
}

\newcommand{\PSPACE}{\text{\normalfont PSPACE}\xspace}
\newcommand{\NP}{\text{\normalfont NP}\xspace}
\newcommand{\problemname}[1]{\textnormal{\textsc{#1}}}
\newcommand{\defproblem}[3]{
  \vspace{1mm}
\begin{center}
\noindent\fbox{

  \begin{minipage}{.9\linewidth}
  \begin{tabular*}{\linewidth}{@{\extracolsep{\fill}}lr} \textsc{#1}   \\ \end{tabular*}
  {\bf{Input:}} #2  \\
  {\bf{Question:}} #3
  \end{minipage}

  }
\end{center}
  \vspace{1mm}
}



\title{On the Majority Game Chromatic Number of Forests and Other Graphs} 

\author{Yash Chawda\footnote{Corresponding author}}{Department of Mathematics, Indian Institute of Technology Jodhpur, NH 62 Nagaur Road, Karwar, Jodhpur 342030, Rajasthan, India \and \url{https://yashu112.github.io}}{p25ma0204@iitj.ac.in}{https://orcid.org/0009-0003-3713-8556}{University Grants Commission (UGC) Junior Research Fellowship (JRF), Govt.\ of India.}

\author{Saraswati Girish Nanoti}{Department of Computer Science and Automation, Indian Institute of Science, C.\,V. Raman Avenue, Bengaluru 560012, Karnataka, India \and \url{https://sites.google.com/iitgn.ac.in/saraswati-girish-nanoti}}{saraswatig@iisc.ac.in}{https://orcid.org/0009-0009-7789-8895}{Anusandhan National Research Foundation (ANRF) National Post-Doctoral Fellowship (N-PDF), Govt.\ of India.}

\author{Brahadeesh Sankarnarayanan}{Department of Mathematics, Indian Institute of Technology Jodhpur, NH 62 Nagaur Road, Karwar, Jodhpur 342030, Rajasthan, India \and \url{https://brahadeesh1994.github.io}}{brahadeesh@iitj.ac.in}{https://orcid.org/0000-0001-9191-1253}{}

\authorrunning{Y. Chawda and S.\,G. Nanoti and B. Sankarnarayanan} 

\Copyright{Yash Chawda and Saraswati Girish Nanoti and Brahadeesh Sankarnarayanan} 



\category{} 

\relatedversiondetails[cite={ChawdaNanotiEtAl2026a}]{arXiv Version}{https://arxiv.org/abs/2609.23803} 

\acknowledgements{This article is dedicated to the memory of Prof.~Subir K.\ Ghosh (1953--2026), a dedicated researcher and mentor in algorithms and computational geometry. His tireless efforts were instrumental in the establishment of the CALDAM conference series, to which he continued to contribute for many years. The second and third authors are particularly grateful to the organizers of the Indo-Polish Pre-Conference School on Algorithms and Combinatorics at CALDAM 2026, where they were introduced to majority colorings.
The second author is also thankful to Arka Ray and Agrim Dewan for helpful discussions about the complexity results.}

\nolinenumbers 

\EventEditors{Raghavan Komondoor, Ramprasad Saptharishi, and James Worrell}
\EventNoEds{3}
\EventLongTitle{46th IARCS Annual Conference on Foundations of Software Technology and Theoretical Computer Science (FSTTCS 2026)}
\EventShortTitle{FSTTCS 2026}
\EventAcronym{FSTTCS}
\EventYear{2026}
\EventDate{December 16--18, 2026}
\EventLocation{IIT Delhi, New Delhi, India}
\EventLogo{}
\SeriesVolume{400}
\ArticleNo{31}

\begin{document}
\maketitle

\begin{abstract}
A \emph{majority coloring} (also called an \emph{unfriendly partition}) of a graph \(G\) is a vertex coloring of \(G\) in which no vertex has more than half of its neighbors colored with its own color.
The least number of colors required for a majority coloring of \(G\) is the \emph{majority chromatic number} \(\mu(G)\).
Lov\'asz~\cite{Lovasz1966} (1966) gave an elegant proof that \(\mu(G) = 2\) for every nonempty finite graph.
The \emph{majority coloring game}, introduced by Bosek--Grytczuk--Jak\'obczak~\cite{BosekGrytczukEtAl2019} (2019), is a two-player Maker-Breaker-type game where the players alternately color vertices while maintaining the majority condition at each vertex; the least number of colors required for the first player to have a winning strategy on \(G\) is the \emph{majority game chromatic number} $\mu_g(G)$.
In contrast with the static case, Bosek--Grytczuk--Jak\'obczak show that $\mu_g(G)$ is unbounded in general, while $\mu_g(G)\le \mathrm{col}_g(G)$, where \(\mathrm{col}_g(G)\) is the game coloring number of \(G\).

It is known (cf.~Faigle--Kern--Kierstead--Trotter~\cite{FaigleKernEtAl1993} (1993)) that, for any acyclic graph \(G\), \(\mathrm{col}_g(G) \leq 4\), so \(\mu_g(G) \leq 4\) as well.
We improve this bound by showing that \(\mu_g(G) \leq 3\) for any acyclic graph \(G\) of maximum degree at most $4$.
We also show that \(\mu_g(G) = 2\) if \(G\) is a nonempty path, star, or complete graph, or a disjoint union of such graphs.
These improve the results of Bosek--Grytczuk--Jak\'obczak.
We also initiate the study of the computational complexity of the pre-coloring extension problem for the static and game versions of majority coloring.
We show that the \problemname{Majority \(k\)-precoloring Extension} problem is \NP-complete for each \(k \geq 2\).
When \(k = 2\), it remains \NP-complete on the class of bipartite graphs where maximum degrees of each part are \(5\) and \(6\), and when \(k = 3\) it remains \NP-complete on the class of \(12\)-regular planar graphs.
We show that the \problemname{Majority Precoloring Extension Game} problem is \PSPACE-complete, even when the number of colors is fixed to be the chromatic number of the input
graph, and the \problemname{Majority \(2\)-Precoloring Extension Game} problem is \PSPACE-complete on the class of bipartite graphs in which one part has maximum degree equal to \(9\).
The results on the game versions also hold for the variation where Bob starts the game.
\end{abstract}

\section{Introduction}\label{S:Introduction}
All graphs considered are simple.
Let $G=(V,E)$ be a graph and \(C\) be a set of colors.
A coloring $c:V\to C$ is a \emph{majority coloring} if, for every vertex $v$, we have
\(
|\{u\in N(v): c(u)=c(v)\}|\le d(v)/2,
\)
that is, at most half the neighbors of \(v\) have the same color as \(v\).
The minimum size of such a color set is the \emph{majority chromatic number} $\mu(G)$. Lov\'asz~\cite{Lovasz1966} gave an elegant proof of the fact that $\mu(G) = 2$ for every finite nonempty graph $G$; in fact, Lov\'asz shows the more general result that every graph has a \(k\)-coloring such that every vertex \(v\) has at most \(\frac{1}{k} \deg(v)\) neighbors of the same color (see also \cite{BorodinKostochka1977,Lawrence1978,Bernardi1987}).
A majority \(2\)-coloring is also called an \emph{unfriendly partition} of the graph, and the Unfriendly Partition Conjecture attributed to  Cowen--Emerson~\cite{CowenEmerson1985} (see~\cite{AharoniMilnerEtAl1990,ShelahMilner1990}) states that \(\mu(G) = 2\) for any nonempty infinite graph as well.
A compactness argument shows that this holds for any infinite graph which is locally finite, and Aharoni--Milner--Prikry~\cite{AharoniMilnerEtAl1990} showed that it also holds if \(G\) is an infinite graph with finitely many vertices of infinite degree.
On the other hand, Shelah--Milner~\cite{ShelahMilner1990} showed that \(\mu(G) \leq 3\) for any graph \(G\), and they also constructed an uncountably infinite graph for which \(\mu(G) = 3\).
However, the conjecture remains open in general for countably infinite graphs: see~\cite{BruhnDiestelEtAl2010,Berger2017,KalinowskiPilsniakEtAl2025} for partial results on several graph classes.

Several variations on majority colorings beyond the basic vertex-coloring setting have been considered in the literature.
Kreutzer--Oum--Seymour--van der Zypen--Wood~\cite{KreutzerOumEtAl2017} showed that every digraph has a \(4\)-coloring such that for every vertex \(v\), at most half the \emph{out-neighbors} of \(v\) have the same color as \(v\), and they conjectured that every digraph is in fact majority \(3\)-colorable.
Majority edge-colorings were defined by Bock--Kalinowski--Pardey--Pils\'niak--Rautenbach--Wo\'zniak~\cite{BockKalinowskiEtAl2023}, and they showed that if \(\delta(G) \geq 2\) then \(\mu'(G) \leq 4\), and if \(\delta(G) \geq 4\), then \(\mu'(G) \leq 3\).
Majority list colorings were studied by Anholcer--Bosek--Grytczuk~\cite{AnholcerBosekEtAl2017} for vertex-colorings, and by P\k{e}ka{\l}a--Przyby{\l}o~\cite{PekalaPrzybylo2025} for edge-colorings.
Anholcer--Bosek--Grytczuk--Gutowski--Przyby{\l}o--Zaj\k{a}c~\cite{AnholcerBosekEtAl2025} studied majority on-line list colorings and further generalizations for both undirected and directed graphs.

Our focus in this paper is on a game-theoretic variant, the \emph{majority coloring game}, which was introduced by Bosek--Grytczuk--Jak\'obczak~\cite{BosekGrytczukEtAl2019}.
Two players, Alice and Bob, alternately color vertices from a fixed palette, maintaining a valid partial majority coloring throughout. Alice wins if all vertices are eventually colored; otherwise Bob wins. The least number of colors on which Alice has a winning strategy is the \emph{majority game chromatic number}, denoted $\mu_g(G)$.
Alice plays first, but we also prove our results for the version where Bob plays first.

In contrast to the static parameter $\mu(G)$, the game parameter $\mu_g(G)$ exhibits significantly richer behavior.
It is unbounded on the class of graphs \(G\) such that \(\chi(G) = 2\) (cf.~\cite[Theorem 1]{BosekGrytczukEtAl2019}), as well as on the class of graphs \(G\) with \(\mathrm{col}(G) = 3\) (cf.~\cite[Theorem 5]{BosekGrytczukEtAl2019}), where \(\mathrm{col}(G)\) is the \emph{coloring number} of \(G\) and is equal to one more than the degeneracy of \(G\).

On the other hand, $\mu_g(G)$ is bounded above by the \emph{game coloring number} of \(G\), $\mathrm{col}_g(G)$, which is defined as follows.
The \emph{marking game} on a graph \(G\) of order \(n\) is a two-player game where the players alternately choose vertices of \(G\), thereby inducing a linear order on \(V(G)\), say \(v_1 \prec v_2 \prec \dotsb \prec v_n\).
The \emph{back degree} of a vertex \(v_i\) is the number of vertices adjacent to \(v_i\) and preceding it in this linear order.
The goal of the first player, Alice, is to minimize the largest back degree in the linear order whereas the goal of the second player, Bob, is to maximize it.
The game coloring number \(\mathrm{col}_g(G)\) is the least \(k\) such that Alice has a strategy for the marking game on \(G\) such that the largest back degree is less than \(k\). 
If Alice follows the marking game strategy in the majority coloring game, then it is clear that fewer than \(\mathrm{col}_g(G)\) colors are forbidden for the vertex that she chooses in each round, so \(\mu_g(G) \leq \mathrm{col}_g(G)\) (cf.~\cite[Theorem 2]{BosekGrytczukEtAl2019}). 

In particular, Faigle--Kern--Kierstead--Trotter~\cite{FaigleKernEtAl1993} showed that $\mathrm{col}_g(G)\le 4$ for any forest \(G\).
Thus, \(\mu_g(G) \leq 4\) for any forest \(G\), and Bosek--Grytczuk--Jak\'obczak~\cite{BosekGrytczukEtAl2019} gave a strategy for Alice to show that in fact $\mu_g(T)\le 3$ for any complete binary tree \(T\).
The present work continues this line of investigation.
In particular, we establish the following:
\begin{restatable}{theorem}{maxdegreetheorem}\label{T:Main1}
    If \(G\) is a forest with \(\Delta(G) \leq 4\), then $\mu_g(G)\le 3$.
    This also holds for the variation where the first move is played by Bob.
\end{restatable}
\begin{restatable}{theorem}{pathstarcomplete}\label{T:Main2}
    If \(G\) is a nonempty path, star, complete graph, or a disjoint union of some copies of these graphs, then \(\mu_g(G) = 2\).
    This also holds for the variation where the first move is played by Bob.
\end{restatable}

We also prove the following complexity results for pre-coloring extensions of majority coloring and its game version.
\begin{restatable}{theorem}{precolextstatic}\label{T:Main3}
    The \problemname{Majority \(k\)-precoloring Extension} problem is \NP-complete for each \(k \geq 2\), and when \(k = 2\) it is \NP-complete even on the class of bipartite graphs in which the two parts have maximum degrees \(5\) and \(6\), respectively.
\end{restatable}
\begin{restatable}{theorem}{precolextthreenpcomplete}\label{T:extmajcolorthree}
    The \problemname{Majority \(3\)-precoloring Extension} problem is \NP-complete on the class of \(12\)-regular planar graphs.
\end{restatable}
\begin{restatable}{theorem}{precolextgame}\label{T:Main4}
    The \problemname{Majority Precoloring Extension Game} problem is \PSPACE-complete, even when the number of colors is fixed to be the chromatic number of the input graph.
    This also holds for the variation where the first move is played by Bob.
\end{restatable}

\begin{restatable}{theorem}{precolextgametwo}\label{T:Main4'}
    The \problemname{Majority \(2\)-precoloring Extension Game} problem is \PSPACE-complete on the class of bipartite graphs in which one part has maximum degree \(9\).
    This also holds for the variation where the first move is played by Bob.
\end{restatable}
To the best of our knowledge, these are the first results on the computational complexity of this parameter, adding to a growing body of results on the complexity of coloring games (see \cite{Bodlaender1991,MarcilonMartinsEtal2020,FraenkelGoldschmidt1987,RahmanWatson2023,SalesMarcilonEtal2026}).

\subsection{Outline of the paper}\label{SS:Outline}

We begin by examining the strategy for Alice as outlined in \cite{BosekGrytczukEtAl2019}, and observe that their strategy can fail on perfect binary trees of sufficiently large depth.
We identify the specific point of failure, and our fix will extend Alice's strategy to all forests of maximum degree at most \(4\), and a variation of this strategy will be used to prove the result for paths.
Our approach is configuration-based: we identify a small family of local obstruction patterns and proceed to construct the configurations that can force a fourth color on a vertex, and show that Alice can play so as to prevent their occurrences.
In the final two sections, we establish the complexity results of Theorem~\ref{T:Main3}--\ref{T:Main4'}.

\section{Counterexample to the strategy proposed in \texorpdfstring{\cite{BosekGrytczukEtAl2019}}{[10]} and some consequences}\label{S:Counterexample}

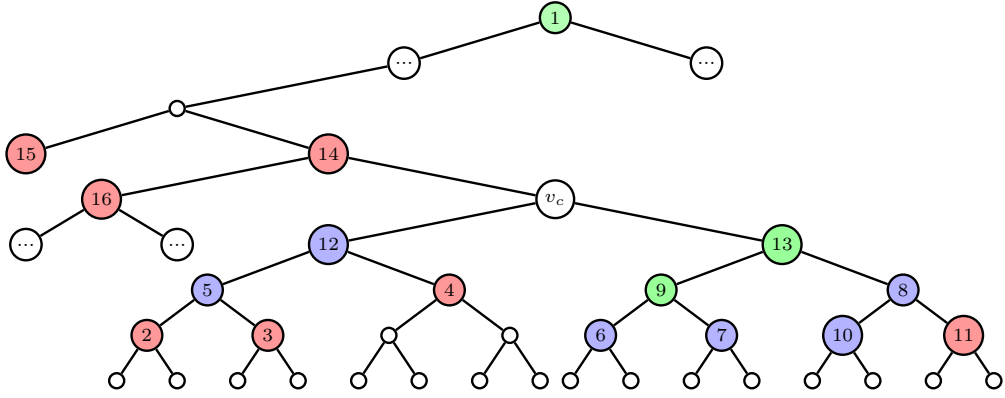
\begin{figure}[ht]     
    \begin{tikzpicture}
        [every node/.style={draw, circle, inner sep=1pt},
        level distance=6mm,
        font=\scriptsize,
        level 1/.style={sibling distance=40mm},
        level 2/.style={sibling distance=50mm},
        level 3/.style={sibling distance=40mm},
        level 4/.style={sibling distance=60mm},
        level 5/.style={sibling distance=32mm},
        level 6/.style={sibling distance=16mm},
        level 7/.style={sibling distance=8mm},
        line width=0.9pt, 
        smooth, scale=0.8]

        \node [fill=green!30, inner sep=2pt] {1}
            child { node [xshift=-30mm] {}
            child { node [fill=red!50] {15}
                child { node [xshift=2cm] {}}
                child { node [xshift=-2cm]{}}}
            child { node [fill=red!50] {14}
            child { 
            node [fill=red!50] {16}
            child { node [xshift=2cm] {}}
            child { node [xshift=-2cm] {}}
            }
        child{
        node [xshift=1cm]{\(v_c\)}
        child { node [fill=blue!30] {12}
        child {node [fill=blue!30] {5}
            child {node [fill=red!50] {2}
                child {node {}}
                child {node {}}}
            child {node [fill=red!50] {3}
                child {node {}}
                child {node {}}}
            }
        child {node [fill=red!50] {4}
            child {node {}
                child {node {}}
                child {node {}}}
            child {node {}
                child {node {}}
                child {node {}}}
            }
        }
        child { node [fill=green!40, inner sep=2pt] {13}
        child {node [fill=green!40, inner sep=2pt] {9}
            child {node [fill=blue!30] {6}
                child {node {}}
                child {node {}}}
            child {node [fill=blue!30] {7}
                child {node {}}
                child {node {}}}
            }
        child {node [fill=blue!30] {8}
            child {node [fill=blue!30] {10}
                child {node {}}
                child {node {}}}
            child {node [fill=red!50] {11}
                child {node {}}
                child {node {}}}
            }
        }
        }}
        }
        child { node [xshift=8mm]  {}}
        ;
    \end{tikzpicture}
    \caption{The vertex \(v_c\) ends up with no available color in this game by following Alice's strategy as outlined in \cite{BosekGrytczukEtAl2019}.
    Note that vertices 1, 9 and 13 are colored green, vertices 2, 3, 4, 11, 14, 15, and 16 are colored red, and vertices 5, 6, 7, 8, 10 and 12 are colored blue.
    Vertices 12, 13 and 14 are majority colored with blue, green and red, respectively.}
    \label{F:counter_example_Bosek}
\end{figure}

In \cite{BosekGrytczukEtAl2019}, the authors propose the following strategy for Alice on a complete binary tree with a palette of \(3\) colors:
\begin{enumerate}
	\item Alice starts by coloring the root.
	\item If Bob colors a vertex, Alice colors its sibling:
	\begin{enumerate}
		\item with the same color if possible, else
		\item with any other available color.
	\end{enumerate}
\end{enumerate}
In particular, they note that if the vertex \(v\) is ``special'', i.e., it is a colored vertex whose sibling and both children are all colored with the same color, then the parent of \(v\) (if it is uncolored) can be colored with the same color as \(v\).

Now, consider a perfect binary tree \(T\) of depth \(\geq 8\) as shown in Figure~\ref{F:counter_example_Bosek}.
The vertices are labeled by the move order, so Alice colors the odd numbered vertices, and Bob colors the even numbered vertices.
Though Alice follows the above strategy in~\cite{BosekGrytczukEtAl2019}, the vertex \(v_c\) has no legal color available after 16 moves, since each of its three neighbors are majority colored with distinct colors, and therefore Bob wins this game on \(T\).
Here and for the rest of paper, for the sake of brevity, we say that a vertex is \emph{majority colored} with a color \(i\), if the corresponding vertex \(v\) is colored with the color \(i\). and exactly \(\floor{\deg(v)/2}\) neighbors of \(v\) are also colored with the same color \(i\).
In particular, Alice's decision to color the vertex \(13\) with the same color as the ``special'' vertex \(9\) causes the vertex \(13\) to be majority colored.
Since its parent, \(v_c\), is uncolored, this move by Alice is helpful to Bob, and indeed Bob manages to exploit it to prevent any legal coloring of \(v_c\).

Thus a winning strategy for Alice must prevent the formation of any such \emph{critical vertices} \(v_c\), i.e., any uncolored vertex for which all available colors violate the majority condition.
Therefore, our strategy for Alice shall immediately neutralize any newly created majority colored vertices by coloring their parents whenever possible; in particular, she must avoid creating any majority colored vertices on her move.

We start with some simple observations about majority colorings and violations of the majority condition in the majority coloring game.
\begin{observation}\label{obs_checkx}
	\hfill
	\begin{enumerate}
		\item\label{obs_checkx_a} Any colored leaf is majority colored.
		\item\label{obs_checkx_b} If coloring a vertex \(v\) on a move violates the majority condition, then any violations occur only at vertices in \(N[v]\), the closed neighborhood of \(v\).
		\item\label{obs_checkx_c} If coloring a vertex \(v\) with the color \(i\) on a move violates the majority condition at \(v\) itself, then instead coloring \(v\) with a color \(j \neq i\) does not violate the majority condition at \(v\) on that move (though it may violate the majority condition at another vertex in \(N(v)\), the open neighborhood of \(v\)).
	\end{enumerate}
\end{observation}

By the above observations, a critical vertex \(v_c\) arises on a tree in the majority \(3\)-coloring game if and only if at least one of the following configurations holds:
\begin{enumerate}
	\item \(v_c\) has three majority colored neighbors with three pairwise distinct colors (see Figure~\ref{configI});
	\item \(v_c\) has two majority colored neighbors with two distinct colors, and \(m+1\) neighbors colored with a third color, where \(m = \floor{\deg(v_c)/2}\) (see Figure~\ref{configII}).
\end{enumerate}

    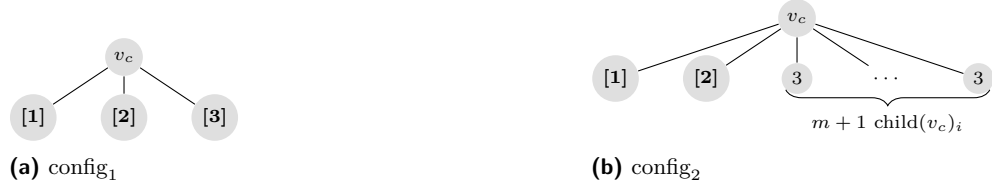
\begin{figure}[!ht]
    \begin{subfigure}{0.45\textwidth}
        \begin{tikzpicture}
            [level distance=8mm,
            font=\scriptsize,
            every node/.style={fill=gray!25,circle, inner sep=2pt},
            level 1/.style={sibling distance=12mm},
            smooth, line width=0.8, scale=0.8]

            \node {$v_c$}
            child{node{\majcol{1}}}
            child{node{\majcol{2}}}
            child{node{\majcol{3}}};
        \end{tikzpicture}
        \caption{\configuration{1}}
        \label{configI}
    \end{subfigure}
    \hfill 
    \begin{subfigure}{0.45\textwidth}
    \begin{tikzpicture}
        [every node/.style={fill=gray!25,circle, inner sep=2pt},
        level distance=8mm,
        font=\scriptsize,
        sibling distance=12mm,
        smooth, line width=0.8, scale=0.8]

        \node (root) {$v_c$}
        child { node (c1) {\majcol{1}} }
        child { node (c2) {\majcol{2}} }
        child { node (c3) {3}}
        child { node (c4) [plain]{$\cdots$} }
        child { node (c5) {3} };

        \draw [decorate,decoration={brace,mirror,amplitude=6pt}]
        (c3.south west) -- (c5.south east)
        node [plain] [midway,below=6pt] {$m+1$ $\text{child}(v_c)_i$};
        \end{tikzpicture}
        \caption{\configuration{2}}
        \label{configII}
    \end{subfigure}
    \caption{Critical configurations to be avoided by Alice.}
    \label{crit_config}
    \end{figure}
    In Figure~\ref{crit_config} and the rest of the paper, by \(\textbf{\majcol{i}}\) we mean that the corresponding vertex \(v\) is majority colored with the color \(i\).

For rooted trees, the configurations can be further divided according to which vertex is the parent of $v_c$, as shown in Figure~\ref{crit_config_trees}.

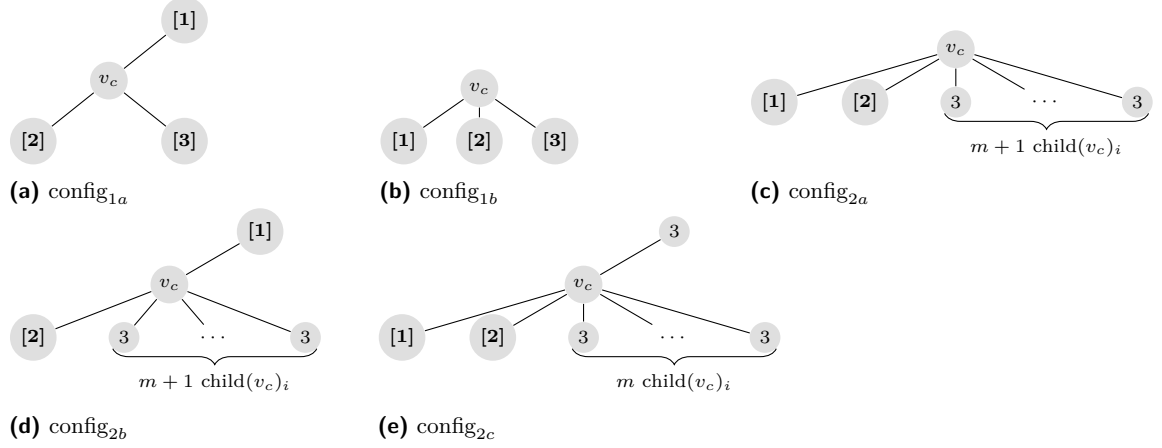
\begin{figure}[!ht]
    \begin{subfigure}[b]{0.3\textwidth}
        \begin{tikzpicture}
            [level distance=8mm,
            font=\scriptsize,
            every node/.style={fill=gray!25,circle, inner sep=2pt},
            level 2/.style={sibling distance=20mm},
            level 3/.style={sibling distance=10mm},
            smooth, line width=0.8, scale=0.8]

            \node {\majcol{1}}
            child{node [xshift=-10mm] {$v_c$}
            child{node {\majcol{2}}}
            child{node {\majcol{3}}}
            };
        \end{tikzpicture}
        \caption{\configuration{1a}}
        \label{configIa}
    \end{subfigure}
    \begin{subfigure}[b]{0.3\textwidth}
        \begin{tikzpicture}
            [level distance=7mm,
            font=\scriptsize,
            every node/.style={fill=gray!25,circle, inner sep=2pt},
            level 1/.style={sibling distance=10mm},
            smooth, line width=0.8]

            \node {$v_c$}
            child{node{\majcol{1}}}
            child{node{\majcol{2}}}
            child{node{\majcol{3}}};
        \end{tikzpicture}
        \caption{\configuration{1b}}
        \label{configIb}
    \end{subfigure}
    \begin{subfigure}[b]{0.3\textwidth}
    \begin{tikzpicture}
        [every node/.style={fill=gray!25,circle, inner sep=2pt},
        level distance=7mm,
        font=\scriptsize,
        sibling distance=12mm,
        smooth, line width=0.8, scale=0.8]

        \node (root) {$v_c$}
        child { node (c1) {\majcol{1}} }
        child { node (c2) {\majcol{2}} }
        child { node (c3) {3}}
        child { node (c4) [plain]{$\cdots$} }
        child { node (c5) {3} };

        \draw [decorate,decoration={brace,mirror,amplitude=6pt}]
        (c3.south west) -- (c5.south east)
        node [plain] [midway,below=6pt] {$m+1$ $\text{child}(v_c)_i$};
        \end{tikzpicture}
        \caption{\configuration{2a}}
        \label{configIIa}
    \end{subfigure}
    \begin{subfigure}[b]{0.34\textwidth}
    \begin{tikzpicture}
        [every node/.style={fill=gray!25,circle, inner sep=2pt},
        level distance=7mm,
        font=\scriptsize,
        sibling distance=12mm,
        smooth, line width=0.8, scale=0.8]

        \node {\majcol{1}}
        child { node [xshift=-12mm] {$v_c$}
        child { node (c1) {\majcol{2}} }
        child { node (c2) {3}}
        child { node (c3) [plain]{$\cdots$} }
        child { node (c4) {3} }
        };

        \draw [decorate,decoration={brace,mirror,amplitude=6pt}]
        (c2.south west) -- (c4.south east)
        node [plain] [midway,below=6pt] {$m+1$ $\text{child}(v_c)_i$};

        \end{tikzpicture}
        \caption{\configuration{2b}}
        \label{configIIb}
    \end{subfigure}
    \begin{subfigure}[b]{0.35\textwidth}
    \begin{tikzpicture}
        [every node/.style={fill=gray!25,circle, inner sep=2pt},
        level distance=7mm,
        font=\scriptsize,
        sibling distance=12mm,
        smooth, line width=0.8]

        \node {3}
        child { node [xshift=-12mm] {$v_c$}
        child { node (c1) {\majcol{1}} }
        child { node (c2) {\majcol{2}} }
        child { node (c3) {3}}
        child { node (c4) [plain]{$\cdots$} }
        child { node (c5) {3} }
        };

        \draw [decorate,decoration={brace,mirror,amplitude=6pt}]
        (c3.south west) -- (c5.south east)
        node [plain] [midway,below=6pt] {$m$ $\text{child}(v_c)_i$};

        \end{tikzpicture}
        \caption{\configuration{2c}}
        \label{configIIc}
    \end{subfigure}
    \caption{Critical configurations on a rooted tree to be avoided by Alice, where \(m = \floor{\deg(v_c)/2}\).}
    \label{crit_config_trees}
\end{figure}

\section{Trees with maximum degree at most \texorpdfstring{\(4\)}{4} are game majority \texorpdfstring{\(3\)}{3}-colorable}

We start by giving a strategy for Alice with \(3\) colors on trees with maximum degree at most \(4\).
We then prove that using this strategy, Alice can avoid some intermediate configurations and as a result she avoids all the critical configurations.

\subsection{Alice's strategy}\label{SS:Strategy}

Alice plays the majority \(3\)-coloring game on a tree \(T\) with \(\Delta(T)\leq 4\) as per the following strategy.
In a rooted tree, denote the parent of the vertex $v$ by \(p(v)\), and the grandparent of \(v\) by \(p^2(v)\).
\begin{itemize}
\item In her first move, Alice colors the an arbitrary vertex and roots the tree at that vertex.
\item Let $u$ be the vertex Bob colors on his last turn. Define the vertex \(v\) as follows:
\(v := p(u)\) if \(p(u)\) is majority colored and \(v := u\) otherwise.

\item Suppose \(p(v)\) is uncolored.
	\begin{itemize}
		\item If \(p^2(v)\) is colored, then Alice colors \(p(v)\) with a color different from \(p^2(v)\).
		\item If \(p^2(v)\) is uncolored, then Alice colors \(p(v)\) such that \(p(v)\) does not become majority colored.
	\end{itemize}
\item In all other cases, Alice colors an uncolored vertex \(x\) with a colored parent, using a color different from \(p(x)\).
\end{itemize}

\begin{algorithm}
\caption{Alice's strategy for the majority \(3\)-coloring game on trees \(T\) with \(\Delta(T) \leq 4\)}\label{A:Algorithm}
\begin{algorithmic}
    \State Color any vertex and root it in the first move.
    \State For the remaining moves, let $u$ be the last vertex colored by Bob.
    \If{$p(u)$ is majority colored}
        \State $v \coloneq p(u)$
    \Else
        \State $v \coloneq u$
    \EndIf
    \If{$p(v)$ is uncolored}
        \If{$p^2(v)$ is colored}
            \State color $p(v)$ different from $p^2(v)$, 
        \ElsIf{$p^2(v)$ is uncolored}
            \State color $p(v)$ such that $p(v)$ is not majority colored
        \EndIf
    \Else
        \State color an uncolored vertex \(x\) with a colored parent, with a color different from \(p(x)\)
    \EndIf
\end{algorithmic}
\end{algorithm}

\subsection{Preliminary lemmas}\label{S:Preliminaries}

We consider three more basic patterns in a rooted tree as shown in Figure~\ref{F:basic_config}.

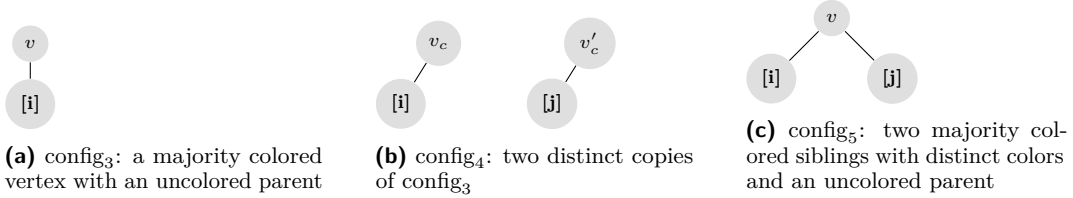
\begin{figure}[ht]

\begin{subfigure}[b]{0.3\textwidth}
 
\begin{tikzpicture}
    [level distance=8mm,
    font=\scriptsize,
    every node/.style={fill=gray!25,circle, inner sep=3pt},
    smooth, line width=0.8]
    \node {$v_c$}
    child{node {\textbf{\majcol{i}}}};
\end{tikzpicture}
\caption{\configuration{3}: a majority colored vertex with an uncolored parent}
\label{config_3}
\end{subfigure}
\hfill
\begin{subfigure}[b]{0.3\textwidth}
    \begin{tikzpicture}
        [level distance=8mm,
        font=\scriptsize,
        every node/.style={fill=gray!25,circle, inner sep=3pt},
        smooth, line width=0.8]

        \node at (0,0) {$v_c$}
        child{node [xshift=-5mm] {\majcol{i}}};
        \node at (2,0) {$v_c'$}
        child{node [xshift=-5mm] {\majcol{j}}};
    \end{tikzpicture}
    \caption{\configuration{4}: two distinct copies of \configuration{3}}
    \label{config_4}
\end{subfigure}
\hfill    
\begin{subfigure}[b]{0.3\textwidth}
 
\begin{tikzpicture}
    [level distance=8mm,
    font=\scriptsize,
    every node/.style={fill=gray!25,circle, inner sep=3pt},
    level 1/.style={sibling distance=16mm},
    smooth, line width=0.8]
    \node {$v_c$}
    child{node{\textbf{\majcol{i}}}}
    child{node{\textbf{\majcol{j}}}};
\end{tikzpicture}
\caption{\configuration{5}: two majority colored siblings with distinct colors and an uncolored parent}
\label{config_5}
\end{subfigure}
\caption{Basic configurations in a rooted tree.}\label{F:basic_config}
\end{figure}

We will be able to show that by the strategy described in Algorithm~\ref{A:Algorithm} Alice will never create \configuration{3}, as required based on the analysis in Section~\ref{S:Counterexample}.
The bound on the maximum degree will be crucial in ruling out the occurrence of any critical configuration.

Throughout, we assume that the game is being played using \(3\) colors on a tree \(T\) with \(\Delta(T) \leq 4\), that Alice follows the strategy in Algorithm~\ref{A:Algorithm}, and that Alice has played the first move and rooted the tree at that vertex.
We shall show by the following circle of lemmas that \configuration{5} does not arise, that Alice does not create \configuration{3} on her move, and that Alice can indeed follow the strategy of Algorithm~\ref{A:Algorithm}.

\begin{lemma}\label{L:config_3}
    Assume that Alice can follow the strategy of Algorithm~\ref{A:Algorithm}.
    Then, Alice does not create \configuration{3}.
\end{lemma}
\begin{proof}
    Algorithm~\ref{A:Algorithm} ensures that when Alice has to color a vertex \(w\) with an uncolored parent, then Alice can choose a color such that \(w\) is not majority colored.
\end{proof}

\begin{lemma}\label{L:config_5}
    Assume that Alice does not create \configuration{3}.
    Then, \configuration{5} (Figure~\ref{config_5}) does not arise.
\end{lemma}
\begin{proof}
    Suppose the conclusion is false, for the sake of contradiction.
    Consider the first instance in which \configuration{5} arises, that is, there is an uncolored vertex \(v_c\) with two majority colored children \majcol{i} and \majcol{j} where \(i \neq j\).
    Since we assume that Alice does not create \configuration{3}, both \majcol{i} and \majcol{j} must have been created by Bob.
    Again, when Bob created the first one between them, say \majcol{i}, Alice would have chosen to color $v_c$ in the very next move as per Algorithm~\ref{A:Algorithm}.
    Since the first instance of \configuration{5} has not yet arisen at this stage, Lemma~\ref{L:algo-works} is applicable, and Alice can indeed color \(v_c\) as per Algorithm~\ref{A:Algorithm}, contradicting that \(v_c\) was uncolored when \majcol{j} was created.
\end{proof}

\begin{lemma}\label{L:algo-works}
    Assume that \configuration{5} does not arise.
    Then Alice can follow Algorithm~\ref{A:Algorithm} to color her choice of vertex.
\end{lemma}
\begin{proof}
    Say Alice chooses to color \(w\) as per Algorithm~\ref{A:Algorithm}.
    First, consider the case when \(p(w)\) is colored, say with color \(1\).
    If the colors \(2\) and \(3\) are forbidden for \(w\), then either \(w\) has two children that are colored \majcol{2} and \majcol{3}, or \(w\) has one child colored \majcol{2} and \(\floor{\deg(w)/2} + 1\) children colored \(3\) (or vice-versa).
    The first scenario is not possible because we assume that \configuration{5} does not arise.
    For the second scenario, a case analysis shows that since \(\deg(w) \leq \Delta(T) \leq 4\), \(w\) does not have enough children to have both a majority colored child with the color \majcol{2} as well as \(\floor{\deg(w)/2} + 1\) children colored \(3\).
    Indeed, for \(\deg(w) = 1,2,3,4\), we find that the second scenario requires \(w\) to have at least \(\floor{\deg(w)/2} + 2 = 2,3,3,4\) children, respectively.
    But this is not possible since \(w\) also has a parent.
    So, at least one of the colors \(2\) and \(3\) is available for Alice at the vertex \(w\).

    Next, consider the case when \(p(w)\) is not colored.
    As per Algorithm~\ref{A:Algorithm}, this situation only arises when the vertex \(w\) has a colored child.
    In particular, \(w\) is not a leaf: this is crucial because if \(w\) were a leaf then by Observation~\ref{obs_checkx}(\ref{obs_checkx_a}), any color for \(w\) would make it majority colored, violating the requirements of Algorithm~\ref{A:Algorithm}.
    Now, by Observation~\ref{obs_checkx}(\ref{obs_checkx_c}), if Alice has at least two choices of colors available for \(w\), then she can choose one for \(w\) such that \(w\) does not become majority colored.
    So, suppose (without loss of generality) that Alice cannot color \(w\) with colors \(2\) and \(3\).
    Then, either \(w\) has two children that are colored \majcol{2} and \majcol{3}, or \(w\) has one child colored \majcol{2} and \(\floor{\deg(w)/2} + 1\) children colored \(3\) (or vice-versa).
    We have seen in the previous paragraph that neither scenario can arise, so at least one of the colors \(2\) and \(3\) is available for Alice at the vertex \(w\).
    Say that it is the color \(3\).
    If coloring \(w\) with \(3\) does not make \(w\) majority colored, then we are done.
    If not, Alice would be forced to use the color \(3\) on \(w\) only if the colors \(1\) and \(2\) are both forbidden at \(w\).
    By the same argument repeated for this pair of colors, at least one of them is available at \(w\), say the color \(2\).
    Recall by Observation~\ref{obs_checkx}(\ref{obs_checkx_c}) that when the color \(3\) makes \(w\) majority colored, the color \(2\) does not make \(w\) majority colored.
    So, Alice can color \(w\) such that \(w\) does not become majority colored.
    This completes the proof.
\end{proof}

\subsection{Proof of Theorem \texorpdfstring{\ref{T:Main1}}{1.1}}

Lemma~\ref{L:config_3} and Lemma~\ref{L:config_5} together imply the following result.

\begin{corollary}\label{C:config_5}
    If Alice can follow the strategy in Algorithm~\ref{A:Algorithm}, then \configuration{5} does not arise.
\end{corollary}

We can now deduce that Algorithm~\ref{A:Algorithm} is well-defined.

\begin{corollary}\label{C:algo_works}
    Alice can follow the strategy of Algorithm~\ref{A:Algorithm}.
\end{corollary}
\begin{proof}
    Suppose not, for the sake of contradiction.
    Consider the first instance where Alice cannot follow the strategy in Algorithm~\ref{A:Algorithm}.
    Say Alice's turn when this occurs is move \(n_0\) of the game, \(n_0 \geq 3\).
    Then, up to move \(n_0 - 2\), Alice could follow the strategy.
    By Lemma~\ref{C:algo_works} and Observation~\ref{obs_checkx}(\ref{obs_checkx_b}), \configuration{5} did not arise in the game up to move \(n_0-1\).
    But then Lemma~\ref{L:algo-works} is applicable at move \(n_0-1\), which means that Alice can follow her strategy on her following move, i.e. move \(n_0\), which is a contradiction.
\end{proof}

We are now ready to prove that Algorithm~\ref{A:Algorithm} defines a winning strategy for Alice on forests \(F\) with \(\Delta(F) \leq 4\).

\maxdegreetheorem*
\begin{proof}
    First, consider the case when the graph is a tree \(T\) with \(\Delta(T) \leq 4\), and assume that Alice starts the game.
    Alice wins if the configurations shown in Figure~\ref{crit_config_trees} do not arise in the game.
    For \configuration{1a} and \configuration{1b}, by Corollary~\ref{C:config_5}, an uncolored $v_c$ with 2 majority colored children does not appear.
    For \configuration{2a}, \configuration{2b} and \configuration{2c} there are $m+1$ colored neighbors of \(v_c\), where \(m = \lfloor \deg(v_c) / 2 \rfloor\), as well as \(2\) other majority colored neighbors.
    So, $\deg(v_c) \geq \lfloor \deg(v_c) / 2 \rfloor+3$. 
    But, as in the proof of Lemma~\ref{L:algo-works}, this is not possible when \(\deg(w) \leq \Delta(T) \leq 4\).
    So, Alice wins with \(4\) colors on any tree \(T\) with \(\Delta(T) \leq 4\).

    Next, the following modification of the strategy works for Alice when the game is played on a tree with Bob to move first: Alice just colors a neighbor of the first vertex colored by Bob using a different color, and then roots the tree at the vertex that she colored. The game then continues following the strategy of Algorithm~\ref{A:Algorithm}.
    By this choice, Alice never creates \configuration{5}, and all the proofs go through verbatim.

    Consequently, this combined strategy applies on any forest \(F\) with \(\Delta(F) \leq 4\), regardless of which player starts the game: Alice just plays as per the above strategy on the tree where Bob played his last move, if possible, otherwise she plays as per Algorithm~\ref{A:Algorithm} on a connected component with the additional case being that if the component is pre-colored, then Alice colors any uncolored vertex \(x\) with a colored parent, with a color different from its parent.
\end{proof}

We also note that since Alice's first move on the forest \(F\) can be arbitrary chosen as per Algorithm~\ref{A:Algorithm}, the strategy extends to forests in which all but one vertex has degree at most \(4\): Alice simply colors the unique exceptional vertex (if it exists) on her first move.

\begin{corollary}
    If \(F\) is a forest in which there is a vertex \(v \in V(F)\) such that \(\deg(w) \leq 4\) for all \(w \neq v\), then \(\mu_g(F) \leq 3\), assuming Alice plays first.
\end{corollary}

\section{Stars, paths, and complete graphs are game majority \texorpdfstring{\(2\)}{2}-colorable}

\subsection{Strategy for stars}\label{SS:star}
Let \(G = K_{1,n}\) be a star, where \(n \geq 2\).
Alice colors the center of \(G\) on the first move, forcing only the second color to be used on the rest of the vertices.
The strategy is similar if Bob goes first: if Bob colors the center in his first move, then the second color is forced on the remaining vertices, and if Bob colors a leaf in his first move, then Alice colors the center on the next move with the other color, and the game continues as before.

\subsection{Strategy for complete graphs}\label{SS:complete}
If \(G = K_n\), where \(n \geq 2\), then Alice starts by coloring any vertex, and in subsequent moves she colors any uncolored vertex with a minimally used color.
Even if Bob goes first, Alice's strategy is to color any uncolored vertex with a minimally used color.

\subsection{Strategy for paths}\label{SS:path}
    The strategy for paths is a slight modification of the strategy for trees of maximum degree at most \(4\) using 2 colors.
    Alice colors a leaf neighbor in her first move and roots the path at that vertex.
    For subsequent moves, she colors the parent \(p(v)\) of the vertex \(v\) last colored by Bob; if \(p(v)\) is colored, she colors the child of the vertex \(v\). In both cases, Alice uses the opposite color to the one used by Bob on \(v\).
    If the child of \(v\) is also colored, Alice colors an uncolored vertex \(x\) closest to the root using a color different from \(p(x)\).

    Now we show that this is a winning strategy for Alice with two colors.
    Consider any arbitrary path.
    Since the degree of any vertex is at most 2, two colors can be forbidden at a given vertex \(v_c\) only if it has two different majority colored neighbors, as shown in Figure~\ref{F:paths_vc}.
    Suppose such a vertex \(v_c\) arises when Alice is playing the game using our strategy.
    
\begin{figure}[!ht]
    \begin{subfigure}[b]{0.48\textwidth}
    \begin{tikzpicture}[
            level distance=8mm,
            font=\scriptsize,
            every node/.style={circle, draw, fill=white, inner sep=2pt},
            line width=0.9pt,
            smooth
        ]
        \node (0) at (-1,0) {...};
        \node (1) at (0,0) [fill=red!50] {\(v_1\)};
        \node (2) at (1,0) [fill=red!50] {\(v_2\)};
        \node (3) at (2,0)  {\(v_c\)};
        \node (4) at (3,0) [fill=blue!50] {\(v_3\)};
        \node (5) at (4,0) [fill=blue!50] {\(v_4\)};
        \node (6) at (5,0) {...};

        \draw (0)--(1)--(2)--(3)--(4)--(5)--(6);
    \end{tikzpicture}
    \caption{}
    \label{F:path_a}
    \end{subfigure}\qquad
    \begin{subfigure}[b]{0.35\textwidth}
    \begin{tikzpicture}[
            level distance=8mm,
            font=\scriptsize,
            every node/.style={circle, draw, fill=white, inner sep=2pt},
            line width=0.9pt,
            smooth
        ]
        \node (0) at (-1,0) {...};
        \node (1) at (0,0) [fill=red!50] {\(v_1\)};
        \node (2) at (1,0) [fill=red!50] {\(v_2\)};
        \node (3) at (2,0)  {\(v_c\)};
        \node (4) at (3,0) [fill=blue!50] {\(v_3\)};

        \draw (0)--(1)--(2)--(3)--(4);
    \end{tikzpicture}
    \caption{}
    \label{F:path_b}
    \end{subfigure}
    \caption{Forbidden configurations in path; \(v_1\) is nearer to the root than other \(v_is\).}
    \label{F:paths_vc}
\end{figure}
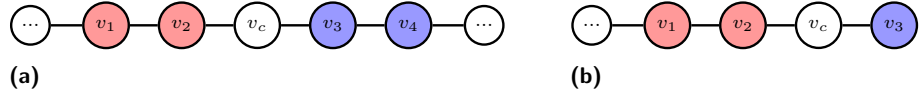

Figure~\ref{F:path_a}:
Consider \(v_2\).
Alice would have colored \(v_2\) if it was the parent of a vertex colored by Bob or the child of a vertex colored by Bob or an uncolored vertex closest to the root.
Since \(v_c\) is uncolored, the first case is not possible.
If Bob colored \(v_1\) before \(v_2\), then Alice would have colored \(v_2\) with a color opposite to that of \(v_1\), which is not desirable for the configuration.
Now, if she colored \(v_2\) because it was the closest uncolored vertex to the root, she would have colored it blue, as \(v_1\) was already red.
Moreover, \(v_2\) cannot be the leaf neighbor that Alice colored in the first move.
If it is, then either \(v_c\) or \(v_1\) has to be a leaf.
But \(v_c\) cannot be leaf as it has degree 2 and if \(v_1\) is a leaf then it cannot have same color as \(v_2\).

So Bob should have colored \(v_2\).
When Bob colored \(v_2\) red, if \(v_1\) was uncolored, Alice would have colored it blue.
If \(v_1\) was colored, Alice would have chosen to color \(v_c\).
If \(v_3\) was not blue, then Alice would have colored \(v_c\) with blue and the configuration does not arise.

Now suppose \(v_3\) was already blue, this means that \(v_3\) got colored before \(v_2\).
Since \(v_c\) is uncolored, \(v_3\) is not the uncolored vertex closest to the root.
So, Alice could have colored \(v_3\) if either \(v_c\) or \(v_4\) was colored red by Bob which is not the case.
So Bob should have colored \(v_3\).
If Bob colored \(v_3\) with blue, then since \(v_2\) was not colored red yet, red was a legal color for \(v_c\) and Alice would have already colored \(v_c\).
Thus the configuration never arises.

Figure~\ref{F:path_b}: Here, \(v_3\) is a leaf.
The same argument holds and the configuration does not arise.

If Bob goes first on a path, then Alice colors a neighbor of the vertex colored by Bob with the opposite color, and roots the path on the vertex she colored. The game continues with the same strategy.

\subsection{Proof of Theorem~\ref{T:Main2}}

On a disjoint union of paths, stars and complete graphs, Alice follows the strategies in Sections~\ref{SS:star}--\ref{SS:path}, preferring to play on the component that Bob last played.
If this is not possible then she chooses to play on any other available component: the only additional requirement is that if she chooses to play on a path that is precolored, then her move is to color an uncolored vertex closest to the root with a color different from its parent.
Thus, we have proved:

\pathstarcomplete*

\section{Complexity of the \problemname{Majority \texorpdfstring{\(k\)}{k}-precoloring Extension} problem}\label{S:Complexity}

We now discuss the complexity results for both the pre-coloring extension problem for majority coloring, and its game version.
Recall that \(\mu(G) = 2\) for every nonempty finite graph~\cite{Lovasz1966}, so deciding whether a graph can be majority \(2\)-colored can be solved in constant time.
So, it is somewhat surprising that the pre-coloring extension problem for majority coloring turns out to be much harder.
We start by defining this problem explicitly.
For each \(k \geq 2\), we define:

\defproblem{Majority \(k\)-precoloring Extension $(G,c^\prime)$}{Graph $G = (V,E)$, a partial function (coloring) $c^\prime \colon V \to [k]$}{Can this partial coloring of $G$ be extended to a majority coloring of $G$ with $k$ colors?}

We will show that the \problemname{Majority \(k\)-precoloring Extension} problem is \NP-complete for each \(k \geq 2\).
Firstly, this problem is in \NP because given a coloring $c$ of the graph $G$, it can be verified in polynomial time whether $c$ satisfies the majority condition, and whether $c(u)=c^\prime(u)$ for vertices $u$ which were assigned some color in the input partial coloring $c^\prime$.

Next, we show that the \problemname{Majority \(2\)-precoloring Extension} problem is \NP-hard by giving a reduction from the \problemname{$3$-SAT} problem, which is known to be \NP-complete~\cite{Karp1972} even when every variables appears exactly twice positively and twice negatively~\cite{BermanKarpinskiEtAl2003}. 

\defproblem{$3$-SAT $(\Phi,m,n)$}{A Boolean formula $\Phi$ in CNF with $m$ clauses and $n$ variables, and each clause containing $3$ literals.}{Does there exist an assignment of each variable $x_i$ to $T/F$ such that $\Phi$ is satisfied?}

Now we describe how to construct an equivalent instance of \problemname{Majority $2$-precoloring Extension} in polynomial time from a given instance of \problemname{\(3\)-SAT}.
For each clause $C_j$ ($j\in [1,m]$), create a vertex in the graph $G$.
We abuse notation and label the vertex corresponding to the clause $C_j$ as $C_j$.
We call these vertices as \emph{clause vertices}.
Add two pendant (degree one) vertices adjacent to each $C_j$ and label these vertices adjacent to $C_j$ as $D^1_j$ and $D^2_j$.
Denote the union of the clause vertices as $C$ and the union of the corresponding pendant vertices as $D$.
Assign the color blue to each $C_j$ from $C$ and the color red to all the $2m$ vertices in $D$.
Now for each variable $x_i$ ($i\in [1,n]$), create a pair of vertices in the graph $G$ and label them $x_i$ and $\bar{x_i}$ (abusing notation).
We call these vertices as \emph{literal vertices} and label their union as $Y$.
For each literal vertex $y_i$ (that is $x_i$ or $\bar{x_i}$), if $y_i$ appears in $k$ clauses, we add $k$ vertices adjacent to $y_i$.
We assign the color red to all these newly added degree one vertices and denote their union by $R^\prime$.
We attach a leaf to each vertex $r^\prime\in R^\prime$, and color it blue, and denote their union by $B^\prime$.
Whenever a literal $y_i$ ($x_i$ or $\bar{x_i}$) appears in the clause $C_j$ in $\Phi$, we add an edge in $G$ from the literal vertex $y_i$ to the clause vertex $C_j$.
Finally, for each pair of literal vertices $x_i$ and $\bar{x_i}$ (corresponding to the same variable and its complement), add a vertex $r_i$ adjacent to both $x_i$ and $\bar{x_i}$ and color it red.
Similarly, add a vertex $b_i$ adjacent to both $x_i$ and $\bar{x_i}$ and color it blue.
Observe that the graph induced by the union of $Y$, $R=\bigcup_{i=1}^{n}\{r_i\}$ and $B=\bigcup_{i=1}^{n}\{b_i\}$ is the disjoint union of $n$ cycles of length $4$, with a pair of non-adjacent vertices of each cycle colored red and blue.
Note that $V(G)=C\cup D\cup Y\cup R^\prime\cup B^\prime\cup R\cup B$ and the only uncolored vertices in the graph $G$ are all the vertices of $Y$.
Observe that the graph $G$ is bipartite with the bi-partition $(C \cup R^\prime \cup R \cup B) \uplus (Y\cup D \cup B^\prime)$.
The description of the reduced instance is now complete, and it can be seen that the reduced instance can be constructed in polynomial time.
Moreover, when each variables appears exactly twice positively and twice negatively, the part $C \cup R^\prime \cup R \cup B$ has maximum degree equal to \(5\) and the part $Y\cup D \cup B^\prime$ has maximum degree equal to \(6\).

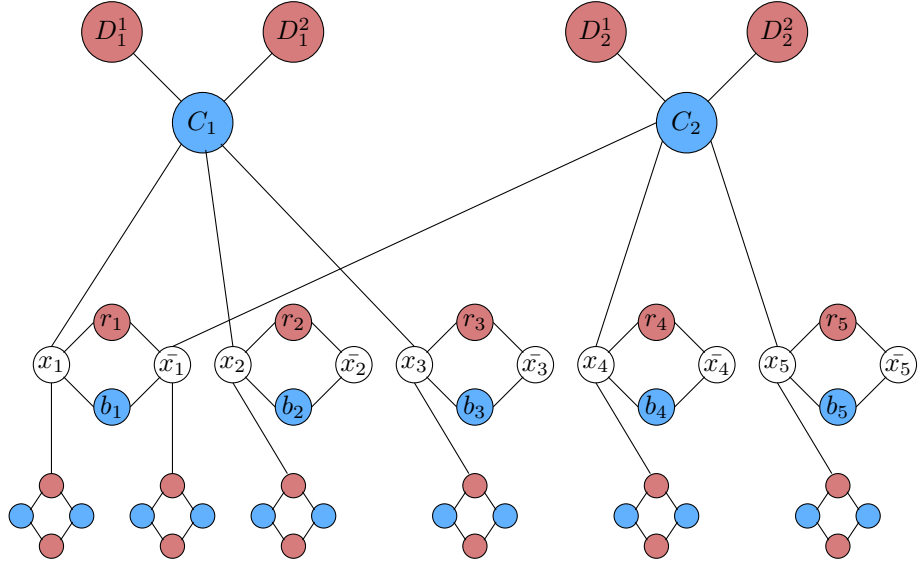
\begin{figure}
    \centering
   \begin{tikzpicture}[scale=0.7,
   smooth, line width=0.8]
       \foreach \x in {0,3,...,12}
         \draw (\x,0) circle (0.3 cm);
    \foreach \x in {0,1,...,4}
      {\pgfmathtruncatemacro{\z}{\x+1}
      \pgfmathtruncatemacro{\y}{3*\x}
      \node at (\y,0){$x_\z$};}
    \foreach \x in {0,3,...,12}
         \draw (\x+2,0) circle (0.3 cm);
    \foreach \x in {0,1,...,4}
      {\pgfmathtruncatemacro{\z}{\x+1}
      \pgfmathtruncatemacro{\y}{3*\x}
      \node at (\y+2,0){$\bar{x_\z}$};}
    \foreach \x in {0,3,...,12}
        {\draw(\x+0.2,0.2)--(\x+0.7,0.7);
        \draw(\x+0.2,-0.2)--(\x+0.7,-0.7);
        \draw(\x+1.8,0.2)--(\x+1.3,0.7);
        \draw(\x+1.8,-0.2)--(\x+1.3,-0.7);
        }
    \foreach \x in {0,3,...,12}
         {\draw[fill=red!50] (\x+1,0.7) circle (0.3 cm);
         \draw[fill=blue!50](\x+1,-0.7) circle (0.3 cm);
         }
    \foreach \x in {0,1,...,4}
      {\pgfmathtruncatemacro{\z}{\x+1}
      \pgfmathtruncatemacro{\y}{3*\x}
      \node at (\y+1,0.7){$r_\z$};
      \node at (\y+1,-0.7){$b_\z$};}

    \draw[fill=red!50](0,-2) circle (0.2 cm);
    \draw[fill=blue!50](-0.5,-2.5) circle (0.2 cm);
    \draw(-0.35,-2.4)--(-0.15,-2.1);

    \draw[fill=red!50](2,-2) circle (0.2 cm);
    \draw[fill=blue!50](1.5,-2.5) circle (0.2 cm);
    \draw(1.65,-2.4)--(1.85,-2.1);

    \draw[fill=red!50](4,-2) circle (0.2 cm);
    \draw[fill=blue!50](3.5,-2.5) circle (0.2 cm);
    \draw(3.65,-2.4)--(3.85,-2.1);

    \draw[fill=red!50](7,-2) circle (0.2 cm);
    \draw[fill=blue!50](6.5,-2.5) circle (0.2 cm);
    \draw(6.65,-2.4)--(6.85,-2.1);
    
     \draw[fill=red!50](10,-2) circle (0.2 cm);
    \draw[fill=blue!50](9.5,-2.5) circle (0.2 cm);     
    \draw(9.65,-2.4)--(9.85,-2.1);

     \draw[fill=red!50](13,-2) circle (0.2 cm);
    \draw[fill=blue!50](12.5,-2.5) circle (0.2 cm);
    \draw(12.65,-2.4)--(12.85,-2.1);
    
    \draw(0,-0.3)--(0,-1.8);
    \draw(2,-0.3)--(2,-1.8);
    \draw(3,-0.3)--(3.9,-1.8);
    \draw(6,-0.3)--(6.9,-1.8);
    \draw(9,-0.3)--(9.9,-1.8);
    \draw(12,-0.3)--(12.9,-1.8);

    \node (c1) at (1.5,2.5) [fill=blue!40, draw, circle, inner sep=1.5pt] {$C_1$};
    \node (c2) at (7,2.5) [fill=blue!40, draw, circle, inner sep=1.5pt]{$C_2$};

    \node (d11) at (0,4) [fill=red!40, draw, circle, inner sep=1.5pt]{$D^1_1$};
    \node (d21) at (3,4) [fill=red!40, draw, circle, inner sep=1.5pt]{$D^2_1$};
    \node (d12) at (5.5,4) [fill=red!40, draw, circle, inner sep=1.5pt]{$D^1_2$};
    \node (d22) at (8.5,4) [fill=red!40, draw, circle, inner sep=1.5pt]{$D^2_2$};

    \draw(c1)--(d11);
    \draw(c2)--(d12);
    \draw(c1)--(d21);
    \draw(c2)--(d22);

    \draw(0,0.3)--(c1);
    \draw(3,0.3)--(c1);
    \draw(6,0.3)--(c1);
    \draw(2,0.3)--(c2);
    \draw(9,0.3)--(c2);
    \draw(12,0.3)--(c2);
    
   \end{tikzpicture}
    \caption{Construction of a reduced instance of \problemname{Majority \(2\)-precoloring Extension} from a given instance of \problemname{$3$-SAT}: $\Phi=C_1\land C_2$ where $C_1=x_1\lor x_2 \lor x_3$ and $C_2=\bar{x_1}\lor x_4 \lor x_5$. The literals $x_1,\bar{x_1},x_2,x_3,x_4,x_5$ appear in exactly one clause and hence the corresponding vertices are adjacent to one red vertex. Note that a satisfying assignment could be setting all the variables to True and the corresponding majority coloring would be to color all the vertices $x_i$ red (for $i\in\{1,\dotsc,5\}$), and color all the vertices $\bar{x_i}$ blue (for $i\in\{1,\dotsc,5\}$).}
    \label{fig:reductionsat}
\end{figure}

Note that the majority condition is always satisfied at the vertices in $D$, because the vertices in $D$ are red and have exactly one neighbor, which is blue. Similarly, each vertex in $R^\prime$ is colored red, has one uncolored neighbor and one blue neighbor, therefore, the majority condition will not be violated at a vertex in $R^\prime$ whether its neighbor is colored red or blue.
The vertices in $B^\prime$ are blue and have one red neighbor each, hence they satisfy the majority condition. 

\begin{lemma}\label{lem:Extended2coloring}
   It is possible to extend the given coloring of $G$ to a majority coloring of $G$ if and only if $\Phi$ has a satisfying assignment.
\end{lemma}

We move the proof of this lemma to Appendix~\ref{AppendixA}. Using this lemma and the fact that the reduction takes polynomial time, we get that the \problemname{Majority \(2\)-precoloring extension} problem is \NP-complete.
For the \problemname{Majority \(k\)-precoloring Extension} problem, where \(k \geq 3\), we show a reduction from \problemname{\(3\)-SAT} as follows: given an instance of \problemname{\(3\)-SAT}, consider the same construction as above, and to each literal vertex add \(k - 2\) leaf vertices that are colored using each color in \(\{3,\dotsc,k\}\) (we let ``red'' denote the color \(1\) and ``blue'' denote the color \(2\)).
The proof of Lemma~\ref{lem:Extended2coloring} goes through for this instance of \problemname{Majority \(k\)-precoloring Extension}, which proves:

\precolextstatic*

To prove Theorem~\ref{T:extmajcolorthree}, we show a reduction from the \problemname{Proper $3$-coloring} problem on \(4\)-regular planar graphs, which is known to be \NP-complete~\cite{Dailey1980}.
We define the more general \problemname{Proper \(k\)-coloring} problem below.

\defproblem{Proper $k$-coloring}{Graph $G$}{Does $G$ have a proper coloring with $k$ colors?}

From a given instance of \problemname{Proper $3$-coloring} we describe how to create an equivalent instance of \problemname{Majority \(3\)-precoloring Extension} in polynomial time.
Given a graph \(G\) of order \(n\) and size \(m\), we create a graph \(H\) of order $n+3m$ and size $6m$ as follows.
For each vertex $v\in G$, create a copy of $v$, denoted by $v^\prime$, in $V(H)$. Denote this set of vertices created by $V^\prime=\{v^\prime:v\in V(G)\}$.
Next, for each edge $uv\in E(G)$, add \(3\) paths of length \(2\) between $u^\prime$ and $v^\prime$ in $H$, and denote the intermediate vertices of these paths by $w^i_{uv}$, for $i\in \{1,2,3\}$.
Now for each $uv\in E(G)$, assign color $i$ to $w^i_{uv}$, \(i \in \{1,2,3\}\) as shown in Figure~\ref{fig:constructionone}.
Let $W = \{w^i_{uv}: uv\in E(G),\,i\in[1,k]\}$.
Note that $V(H)=V^\prime \cup W$.
The construction of the reduced instance is now complete.
Note that the graph $H$ has $n+3m$ vertices (out of which $3m$ vertices are pre-colored) and $6m$ edges; moreover, the graph $H$ can be constructed in polynomial time from the input graph $G$.

\begin{figure}[!ht]
    \centering
    \begin{tikzpicture}[scale=0.8,
    smooth, line width=0.8] 
    \draw(0,0) circle (0.2 cm);
    \node at (0,-0.5){\small${u}$};
    \draw(1,0) circle (0.2 cm);
    \node at (1,-0.5){\small${v}$};
    \draw(0.2,0)--(0.8,0);
    \draw(4,0) circle (0.2 cm);
    \node at (4,-0.5){\small${u^\prime}$};
    \draw(6,0) circle (0.2 cm);
    \node at (6,-0.5){\small${v^\prime}$};
    \draw(4.2,0)--(4.8,0);
    \draw[fill=blue!50](5,0) circle (0.2 cm);
    \draw (5.2,0)--(5.8,0);
    \node at (5,-0.5){\small${w^2_{uv}}$};
    \draw[fill=red!50](5,1) circle (0.2 cm);
    \node at (5,0.5){\small${w^1_{uv}}$};
    \draw[fill=green!50](5,-1) circle (0.2 cm);
    \node at (5,-1.5){\small${w^3_{uv}}$};
    \draw (4.15,0.15)--(4.85,0.85);
    \draw (5.15,0.85)--(5.85,0.15);
    \draw (4.15,-0.15)--(4.85,-0.85);
    \draw (5.15,-0.85)--(5.85,-0.15);
   \end{tikzpicture}
    \caption{Construction of the three length-two paths in the reduced instance. Note that coloring $u^\prime$ and $v^\prime$ with the same color $i$, violates the majority condition at $w^i_{uv}$.}
    \label{fig:constructionone}
\end{figure}

\begin{lemma}\label{lem:Extended3coloring}
 $G$ can be properly colored with $3$ colors if and only if the given partial coloring of $H$ can be extended to a majority coloring of $H$.
\end{lemma}
\begin{proof}
Consider a proper coloring of $G$.
Color each vertex $v^\prime \in V^\prime \subset V(H)$ with the color of the corresponding vertex $v$ in $G$.
This is indeed a majority coloring of $H$, for each vertex $v^\prime\in V^\prime$ has degree (in $H$) equal to $3\cdot \deg_G(v)$ and has exactly $\frac{1}{3}$ of its neighbors colored with color $i$, (for $i\in\{1,2,3\})$.
Therefore, the majority condition is satisfied at all the vertices of $V^\prime$.
Now, consider a vertex $w^{i}_{uv}\in W$, which has two neighbors $u^\prime$ and $v^\prime$.
Note that $u^\prime$ and $v^\prime$ have different colors, because the corresponding vertices $u$ and $v$ have different colors in a proper coloring of $G$.
Therefore, at most one of the two neighbors of $w^{i}_{uv}$ has the same color as $w^{i}_{uv}$, i.e., the color $i$.
Hence, the majority condition is satisfied at each vertex in $W$ as well, and thus at all the vertices of $H$.

Now, we show that whenever $H$ has a majority coloring with $3$ colors which is an extension of the given partial coloring, $G$ has a proper coloring with $3$ colors.
For each vertex $v\in V(G)$, assign it the color of the corresponding vertex $v^\prime\in V^\prime\subset V(H)$.
Suppose that there are two vertices $u$ and $v$ in $G$ which get assigned the same color, say $i$, due to this coloring.
This is possible only when the corresponding vertices $u^\prime$ and $v^\prime$ are assigned the same color $i$ in the coloring of $H$.
But this violates the majority condition at $w^{i}_{uv}$, hence we have a contradiction. Thus, the described coloring is a proper $3$-coloring of $G$.
\end{proof}

Moreover, if the instance \(G\) is a \(4\)-regular planar graph, then the graph \(H\) constructed as above is a \(12\)-regular planar graph.
Thus, we have:

\precolextthreenpcomplete*

\section{Complexity of the \problemname{Majority Precoloring Extension Game} problem}\label{S:Complexity2}
In this section, we show that it is \PSPACE-complete to determine if Alice has a winning strategy with $k\geq2$ colors for a given partially colored graph $G$ in the majority coloring game. We formally define the decision version of this problem below.

\defproblem{Majority Precoloring Extension Game $(G,k,c')$}{Integer $k\geq 2$ and graph $G = (V,E)$ with a partial function (coloring) $c' \colon V \to [k]$}{Does Alice have a winning strategy in the majority coloring game on this partially colored graph $G$ using $k$ colors?}

It can be seen that the \problemname{Majority Precoloring Extension Game} is in \PSPACE because the maximum number of moves that the game can last for is $\lvert V(G) \rvert = n$ (that is, it is polynomially bounded by the input size), hence, the game tree is of polynomial depth. Also, in each turn there are at most $k\cdot n$ possible moves. 

Next, we show that the game is \PSPACE-hard by a reduction from the \problemname{Proper Coloring Game}, which was shown to be \PSPACE-complete in \cite{CostaPessoaEtAl2020}, and which we define below.

\defproblem{Proper Coloring Game $(G,k)$}{Integer $k \geq 2$ and graph $G = (V,E)$}{Does Alice have a winning strategy on the proper coloring game on \(G\) with $k$ colors?}

From a given instance of the \problemname{Proper Coloring Game}, we construct an equivalent instance of \problemname{Majority Precoloring Extension Game} in polynomial time.
The construction is identical to the one in the proof of Theorem~\ref{T:extmajcolorthree}, except instead of adding three paths of length two between the copies in $H$ of each pair of adjacent vertices in $G$, we add $k$ paths of length two and label the $k$ intermediate vertices as $w^{1}_{uv}, w^{2}_{uv}, \dotsc, w^{k}_{uv}$.
We assign the color $i$ to each $w^{i}_{uv}$ for each $i\in[1,k]$.
Note that if \(G\) has order \(n\) and size \(m\), then the graph \(H\) has order \(n + km\) and size \(2km\); moreover, the graph $H$ can be constructed in polynomial time from the input graph $G$.

Now we show that Alice has a winning strategy in the proper coloring game on $G$ with \(k\) colors if and only if she has a winning strategy in the majority precoloring extension game on $H$ with \(k\) colors.
The proof of this is moved to Appendix~\ref{AppendixB}.
Using this and the fact that the construction of the reduced instance takes polynomial time, we get that the \problemname{Majority Precoloring Extension Game} problem is \PSPACE-complete.
Additionally, the reduction goes through for the variant where Bob starts the game, and Lima--Marcilon--Martins--Sampaio~\cite{LimaMarcilonEtAl2022} showed that this variant of the \problemname{Proper Coloring Game} problem is \PSPACE-hard as well.
Note that the number of colors, \(k\), is a part of the input to these problems.
But, it is also proved in \cite{CostaPessoaEtAl2020,LimaMarcilonEtAl2022} that the \problemname{Proper Coloring Game} problem is \PSPACE-complete when the number of colors \(k\) is fixed to be \(\chi(G)\).
So, we get the analogous result for the \problemname{Majority Precoloring Extension Game} problem as well, thereby proving:

\precolextgame*

We now show that the \problemname{Majority Precoloring Extension Game} problem is \PSPACE-complete even when the number of colors is fixed to be \(k=2\).
We shall call this the \problemname{Majority \(2\)-Precoloring Extension Game} problem.
This problem is in \PSPACE by the same argument applied at \(k = 2\), and we show that it is \PSPACE-hard by a reduction from \problemname{POS-CNF}, which is known to be \PSPACE-complete~\cite{Schaefer1978} even when each clause has exactly $5$ literals~\cite{Koepke2025}.

\defproblem{POS-CNF \((\Phi,m,n)\)}{A positive Boolean formula \(\Phi\) with \(m\) clauses and \(n\) variables}{Does Alice have a winning strategy on $\Phi$?}

To elaborate, in the POS-CNF game, there is a Boolean formula $\Phi$ in CNF, and all the variables appear in positive form. 
Each player on their turn picks a variable which is not yet assigned a value from $\{T/F\}$ and sets it to either True or False.
The game ends when all the variables have been assigned some value.
Alice wins if the formula $\Phi$ evaluates to True in the end and Bob wins if it evaluates to False in the end.
Since all the variables appear in positive form, it is never beneficial for Alice to set a variable to False, and it is never beneficial for Bob to set a variable to True.
Hence, without loss of generality, we can assume that in her chance Alice sets a variable to True, and in his chance Bob sets a variable to False.
Furthermore, it was observed in~\cite{CostaPessoaEtAl2020} that if a player has a winning strategy for the POS-CNF game, then that player also has a winning strategy when the other player is allowed to skip turns.

Now we describe how to construct an equivalent instance of the \problemname{Majority $2$-precoloring Extension Game} from a given instance of \problemname{POS-CNF} in polynomial time.
For each clause $C_j$ ($j\in [1,m]$), create a vertex in the graph $G$.
We abuse notation and label the vertex corresponding to the clause $C_j$ as $C_j$.
We call these vertices as \emph{clause vertices}.
If the clause $C_j$ has $k$ literals, add $k-1$ pendant (degree one) vertices adjacent to each $C_j$ and label these vertices adjacent to $C_j$ as $D^1_j,D^2_j,\ldots,D^{k-1}_j$.
Denote the union of the clause vertices as $C$ and the union of the corresponding pendant vertices as $D$.
Assign the color blue to each $C_j$ from $C$ and the color red to all the vertices (at most $5m$) in $D$.

Now for each variable $x_i$ ($i\in [1,n]$), create a pair of vertices in the graph $G$ and label them $x_i$ and $\bar{x_i}$ (abusing notation).
We call the vertices $x_i$ (for $i\in [1,n]$) as \emph{positive vertices} and denote their union as $X$.
We call the vertices $\bar{x_i}$ (for $i\in [1,n]$) as \emph{negative vertices} and denote their union as $L$.
The rest of the construction is the same as the construction in Theorem~\ref{T:Main3}, except note that since $\Phi$ contains only positive literals, there are no edges between the vertices in $L$ and the vertices in $C$; also each vertex in $L$ just has one red neighbor in $R$ and one blue neighbor in $B$.

The description of the reduced instance is now complete and it can be seen that the reduced instance can be constructed from the given instance of \problemname{POS-CNF} in polynomial time.
Observe that the constructed graph $G$ is bipartite, with the bipartition $(X\cup L\cup B^\prime \cup D)\uplus (C\cup R\cup B\cup R^\prime)$;
moreover, if every clause has exactly \(5\) literals, then the part \(C\cup R\cup B\cup R^\prime\) has maximum degree equal to \(9\).

\begin{lemma}\label{lem:Alicewin}
    If Alice has a winning strategy in the POS-CNF game on $\Phi$, then she has a winning strategy in the majority $2$-precoloring extension game on \(G\).
\end{lemma}

\textbf{Strategy for Alice:}
\begin{itemize}
    \item If Alice's winning strategy in the POS-CNF game sets the variable $x_i$ to True, color the positive vertex $x_i$ red.
    \item If Bob colors $x_k$ while $\bar{x_k}$ is uncolored, or vice-versa, assume that Bob has set the variable $x_k$ to False and proceed (according to the strategy in the POS-CNF game).
    \item If Bob colors $x_k$ while $\bar{x_k}$ is colored, or vice-versa, assume that Bob has skipped his turn in the POS-CNF game and proceed (according to the strategy in the POS-CNF game).
    \item If none of the above is possible, color an uncolored vertex $x_i$ (or $\bar{x_i}$) with the color opposite to the vertex $\bar{x_i}$ (or $x_i$).
    Note that if both $x_i$ and $\bar{x_i}$ were not colored at this point, then the variable $x_i$ would be unassigned in the POS-CNF game, and then Alice would have a valid move according to the strategy for the POS-CNF game.
\end{itemize}
The proof that this is actually a winning strategy for Alice is deferred to Appendix~\ref{AppendixC}.

\begin{figure}[!ht]
\begin{subfigure}{\linewidth}
\centering
   \begin{tikzpicture}[scale=0.7,
   smooth, line width=0.8]
       \foreach \x in {0,3,...,9}
         \draw (\x,0) circle (0.3 cm);
    \foreach \x in {0,1,...,3}
      {\pgfmathtruncatemacro{\z}{\x+1}
      \pgfmathtruncatemacro{\y}{3*\x}
      \node at (\y,0){$x_\z$};}
    \foreach \x in {0,3,...,9}
         \draw (\x+2,0) circle (0.3 cm);
    \foreach \x in {0,1,...,3}
      {\pgfmathtruncatemacro{\z}{\x+1}
      \pgfmathtruncatemacro{\y}{3*\x}
      \node at (\y+2,0){$\bar{x_\z}$};}
    \foreach \x in {0,3,...,9}
        {\draw(\x+0.2,0.2)--(\x+0.7,0.7);
        \draw(\x+0.2,-0.2)--(\x+0.7,-0.7);
        \draw(\x+1.8,0.2)--(\x+1.3,0.7);
        \draw(\x+1.8,-0.2)--(\x+1.3,-0.7);
        }
    \foreach \x in {0,3,...,9}
         {\draw[fill=red!50] (\x+1,0.7) circle (0.3 cm);
         \draw[fill=blue!50](\x+1,-0.7) circle (0.3 cm);
         }
    \foreach \x in {0,1,...,3}
      {\pgfmathtruncatemacro{\z}{\x+1}
      \pgfmathtruncatemacro{\y}{3*\x}
      \node at (\y+1,0.7){$r_\z$};
      \node at (\y+1,-0.7){$b_\z$};}

    \draw[fill=red!50](0,-2) circle (0.2 cm);
    \draw[fill=blue!50](-0.5,-2.5) circle (0.2 cm);
    \draw(-0.35,-2.4)--(-0.15,-2.1);
    
   \draw[fill=red!50](3,-2) circle (0.2 cm);
    \draw[fill=blue!50](2.5,-2.5) circle (0.2 cm);
    \draw(2.65,-2.4)--(2.85,-2.1);

    \draw[fill=red!50](5,-2) circle (0.2 cm);
    \draw[fill=blue!50](4.5,-2.5) circle (0.2 cm);
    \draw(4.65,-2.4)--(4.85,-2.1);
    
    \draw[fill=red!50](7,-2) circle (0.2 cm);
    \draw[fill=blue!50](6.5,-2.5) circle (0.2 cm);
    \draw(6.65,-2.4)--(6.85,-2.1);

     \draw[fill=red!50](10,-2) circle (0.2 cm);
    \draw[fill=blue!50](9.5,-2.5) circle (0.2 cm);     
    \draw(9.65,-2.4)--(9.85,-2.1);

    \draw(0,-0.3)--(0,-1.8);
    \draw(3,-0.3)--(3,-1.8);
    \draw(6,-0.3)--(6.9,-1.8);
     \draw(6,-0.3)--(5.1,-1.8);
    \draw(9,-0.3)--(9.9,-1.8);
   
    \node (c1) at (0,2) [fill=blue!50, draw, circle, inner sep=1.5pt] {$C_1$};
    \node (c2) at (3.5,2) [fill=blue!50, draw, circle, inner sep=1.5pt]{$C_2$};
    \node (c3) at (7.5,2) [fill=blue!50, draw, circle, inner sep=1.5pt] {$C_3$}; 
    \node (d12) at (3.5,3.5) [fill=red!50, draw, circle, inner sep=1.5pt]{$D^1_2$};
    \node (d13) at (7.5,3.5) [fill=red!50, draw, circle, inner sep=1.5pt]{$D^1_3$};

    \draw(c2)--(d12);
    \draw(c3)--(d13);

    \draw(0,0.3)--(c1);
    \draw(3,0.3)--(c2);
    \draw(6,0.3)--(c3);
    \draw(6,0.3)--(c2);
    \draw(9,0.3)--(c3);
   \end{tikzpicture}
   \caption{A reduced instance from the instance of POS-CNF $\Phi=C_1\land C_2 \land C_3$, where $C_1=x_1$, $C_2=x_2\lor x_3$ and $C_3=x_3\lor x_4$. Bob has the following winning strategy: If Alice does not set $x_1$ to $T$, set $x_1$ to $F$. If Alice sets $x_1$ to $T$, set $x_3$ to $F$. If Alice sets $x_2$ to $T$, set $x_4$ to $F$. Else, set $x_2$ to $F$.}
\end{subfigure}
\begin{subfigure}{\linewidth}
    \centering
   \begin{tikzpicture}[scale=0.7,
   smooth, line width=0.8]
       \foreach \x in {0,5,8,9}
         \draw[fill=red!50] (\x,0) circle (0.3 cm);
      \foreach \x in {2,3,11}
         \draw[fill=blue!50] (\x,0) circle (0.3 cm);
    \draw (6,0) circle (0.3 cm);
    \foreach \x in {0,1,...,3}
      {\pgfmathtruncatemacro{\z}{\x+1}
      \pgfmathtruncatemacro{\y}{3*\x}
      \node at (\y,0){$x_\z$};}
      \foreach \x in {0,1,...,3}
      {\pgfmathtruncatemacro{\z}{\x+1}
      \pgfmathtruncatemacro{\y}{3*\x}
      \node at (\y+2,0){$\bar{x_\z}$};}
    \foreach \x in {0,3,...,9}
        {\draw(\x+0.2,0.2)--(\x+0.7,0.7);
        \draw(\x+0.2,-0.2)--(\x+0.7,-0.7);
        \draw(\x+1.8,0.2)--(\x+1.3,0.7);
        \draw(\x+1.8,-0.2)--(\x+1.3,-0.7);
        }
    \foreach \x in {0,3,...,9}
         {\draw[fill=red!50] (\x+1,0.7) circle (0.3 cm);
         \draw[fill=blue!50](\x+1,-0.7) circle (0.3 cm);
         }
    \foreach \x in {0,1,...,3}
      {\pgfmathtruncatemacro{\z}{\x+1}
      \pgfmathtruncatemacro{\y}{3*\x}
      \node at (\y+1,0.7){$r_\z$};
      \node at (\y+1,-0.7){$b_\z$};}

    \draw[fill=red!50](0,-2) circle (0.2 cm);
    \draw[fill=blue!50](-0.5,-2.5) circle (0.2 cm);
    \draw(-0.35,-2.4)--(-0.15,-2.1);
    
    \draw[fill=red!50](3,-2) circle (0.2 cm);
    \draw[fill=blue!50](2.5,-2.5) circle (0.2 cm);
    \draw(2.65,-2.4)--(2.85,-2.1);

    \draw[fill=red!50](5,-2) circle (0.2 cm);
    \draw[fill=blue!50](4.5,-2.5) circle (0.2 cm);
    \draw(4.65,-2.4)--(4.85,-2.1);
    
    \draw[fill=red!50](7,-2) circle (0.2 cm);
    \draw[fill=blue!50](6.5,-2.5) circle (0.2 cm);
    \draw(6.65,-2.4)--(6.85,-2.1);

     \draw[fill=red!50](10,-2) circle (0.2 cm);
    \draw[fill=blue!50](9.5,-2.5) circle (0.2 cm);     
    \draw(9.65,-2.4)--(9.85,-2.1);

    \draw(0,-0.3)--(0,-1.8);
    \draw(3,-0.3)--(3,-1.8);
    \draw(6,-0.3)--(6.9,-1.8);
    \draw(6,-0.3)--(4.9,-1.8);
    \draw(9,-0.3)--(9.9,-1.8);
   
   \node (c1) at (0,2) [fill=blue!50, draw, circle, inner sep=1.5pt] {$C_1$};
    \node (c2) at (3.5,2) [fill=blue!50, draw, circle, inner sep=1.5pt]{$C_2$};
    \node (c3) at (7.5,2) [fill=blue!50, draw, circle, inner sep=1.5pt] {$C_3$}; 
    \node (d12) at (3.5,3.5) [fill=red!50, draw, circle, inner sep=1.5pt]{$D^1_2$};
    \node (d13) at (7.5,3.5) [fill=red!50, draw, circle, inner sep=1.5pt]{$D^1_3$};

    \draw(c2)--(d12);
    \draw(c3)--(d13);

    \draw(0,0.3)--(c1);
    \draw(3,0.3)--(c2);
    \draw(6,0.3)--(c3);
    \draw(6,0.3)--(c2);
    \draw(9,0.3)--(c3);
\end{tikzpicture}
    \caption{Example of a game-play where Alice plays by the above strategy: Alice colors $x_1$ red, Bob colors $\bar{x_3}$ red, Alice colors $x_4$ red. Then Bob colors $\bar{x_2}$ red, Alice colors $\bar{x_1}$ blue, Bob colors $x_2$ blue, Alice colors $\bar{x_4}$ blue. Note that coloring $x_3$ blue violates the majority condition at $C_2$ and coloring $x_3$ red violates the majority condition at $r_3$. Hence, $x_3$ cannot be assigned any color and Bob wins.}
\end{subfigure}
    \caption{Demonstration of a game play when Bob has a winning strategy.}
    \label{fig:reductioncnfBob}
\end{figure}

\begin{lemma}\label{lem:Bobwin}
   If Bob has a winning strategy in the POS-CNF game on $\Phi$, then he has a winning strategy in the majority $2$-precoloring extension game on $G$. 
\end{lemma}

\textbf{Strategy for Bob:}
\begin{itemize}
    \item If Bob's winning strategy in the POS-CNF game sets the variable $x_i$ to False, color the negative vertex $\bar{x_i}$ red.
    \item If Alice colors $x_k$ while $\bar{x_k}$ is uncolored, or vice-versa, assume that Alice has set the variable $x_k$ to True and proceed (according to the strategy in the POS-CNF game).
    \item If Alice colors $x_k$ while $\bar{x_k}$ is colored, or vice-versa, assume that Alice has skipped her turn in the POS-CNF game and proceed (according to the strategy in the POS-CNF game).
    \item If none of the above is possible, color an uncolored vertex $x_i$ (or $\bar{x_i}$) with the color opposite to the vertex $\bar{x_i}$ (or $x_i$).
    Note that if both $x_i$ and $\bar{x_i}$ were not colored at this point, then the variable $x_i$ would be unassigned in the POS-CNF game, and then Bob would have a valid move according to the strategy for the POS-CNF game.
\end{itemize}

We shift the proof that this is actually a winning strategy for Bob to Appendix~\ref{AppendixC}.
Hence, by Lemmas~\ref{lem:Alicewin} and~\ref{lem:Bobwin} and the fact that the construction of the reduced instance takes polynomial time, we get that the \problemname{Majority $2$-precoloring Extension Game} is \PSPACE-hard.
Moreover, the above reduction goes through for the variant where Bob starts the game as well, and it was shown in~\cite{Koepke2025} that \problemname{POS-CNF} is \PSPACE-complete for this variant.
Thus, we get:

\precolextgametwo*

\section{Concluding remarks}

We have shown that $\mu_g(G)\le 3$ for all forests with \(\Delta(G) \leq 4\) via a configuration-based strategy.
The approach thus extends the result of Bosek--Grytczuk--Jak\'obczak~\cite{BosekGrytczukEtAl2019} for complete binary trees to, for instance, all binary and ternary trees.
In a recent work~\cite{ChawdaNanotiEtAl2026b}, this upper bound has also been proved for \(2\)-caterpillars, rooted trees with all leaves at depth \(k \leq 4\), and subdivided graphs.
We list several natural directions for further work below.
It is observed in~\cite{BosekGrytczukEtAl2019} that it is not clear whether \(\mu_g(T) \leq 3\) holds for all trees \(T\). We ask:

\begin{question} For which trees \(T\) is it true that $\mu_g(T) = 2$? \end{question}

We have shown that nonempty paths and stars are game majority \(2\)-chromatic, but it is easy to find other small examples of trees that are game majority \(2\)-chromatic.

Next, it is known~\cite{Sidorowicz2007} that \(\mathrm{col}_g(G) \leq 5\) for any cactus graph \(G\).
More generally, it is shown in~\cite{JunoszaSzaniawskiRozej2010} that if \(c\) is a constant such that every edge of a graph \(G\) lies on at most \(c\) cycles, then \(\mathrm{col}_g(G) \leq c + 4\).
Since \(\mu_g(G) \leq \mathrm{col}_g(G)\) as shown in~\cite{BosekGrytczukEtAl2019}, it follows that $\mu_g(G)\le 5$ for any cactus $G$. We ask whether this bound is tight:

\begin{question}
    Is there a cactus graph \(G\) for which \(\mu_g(G) = 5\)?
\end{question}

Lastly, we highlight that the behavior of the majority coloring parameter remains to be explored under variants of the coloring game, such as under restricted color reuse~\cite{ChenSchelpEtAl1997,JanczewskiObszarskiEtAl2022} or alternative majority conditions~\cite{BockKalinowskiEtAl2023}.

{
\raggedright
\bibliography{fsttcs-final}

@article{AharoniMilnerEtAl1990,
author  = {Aharoni, R. and Milner, E. C. and Prikry, K.},
  title   = {Unfriendly partitions of a graph},
  journal = {J. Combin. Theory Ser. B},
  volume  = {50},
  number  = {1},
  pages   = {1--10},
  year    = {1990},
  doi     = {10.1016/0095-8956(90)90092-E}
}

@article{AnholcerBosekEtAl2017,
author  = {Anholcer, M. and Bosek, B. and Grytczuk, J.},
  title   = {Majority choosability of digraphs},
  journal = {Electron. J. Combin.},
  volume  = {24},
  number  = {3},
  pages   = {P3.57},
  year    = {2017},
  doi     = {10.37236/6923}
}

@article{AnholcerBosekEtAl2025,
author  = {Anholcer, M. and Bosek, B. and Grytczuk, J. and Gutowski, G. and Przyby{\l}o, J. and Zaj{\k{a}}c, M.},
  title   = {Mrs. {Correct} and majority colorings},
  journal = {Discrete Math.},
  volume  = {348},
  number  = {11},
  pages   = {114577},
  year    = {2025},
  doi     = {10.1016/j.disc.2025.114577}
}

@article{Berger2017,
author  = {Berger, E.},
  title   = {Unfriendly partitions for graphs not containing a subdivision of an infinite clique},
  journal = {Combinatorica},
  volume  = {37},
  number  = {2},
  pages   = {157--166},
  year    = {2017},
  doi     = {10.1007/s00493-015-3261-1}
}

@techreport{BermanKarpinskiEtAl2003,
  author      = {Berman, P. and Karpinski, M. and Scott, A. D.},
  title       = {Approximation hardness of short symmetric instances of {MAX-3SAT}},
  institution = {Electronic Colloquium on Computational Complexity},
  number      = {TR03-049},
  year        = {2003},
  url         = {https://eccc.weizmann.ac.il/report/2003/049/}
}

@article{Bernardi1987,
author  = {Bernardi, C.},
  title   = {On a theorem about vertex colorings of graphs},
  journal = {Discrete Math.},
  volume  = {64},
  number  = {1},
  pages   = {95--96},
  year    = {1987},
  doi     = {10.1016/0012-365X(87)90243-3}
}

@article{BockKalinowskiEtAl2023,
author  = {Bock, F. and Kalinowski, R. and Pardey, J. and Pil{\'s}niak, M. and Rautenbach, D. and Wo{\'z}niak, M.},
  title   = {Majority edge-colorings of graphs},
  journal = {Electron. J. Combin.},
  volume  = {30},
  number  = {1},
  pages   = {P1.42},
  year    = {2023},
  doi     = {10.37236/11291}
}

@incollection{Bodlaender1991,
author    = {Bodlaender, H. L.},
  title     = {On the complexity of some coloring games},
  booktitle = {Graph-Theoretic Concepts in Computer Science},
  editor    = {M{\"o}hring, R. H.},
  series    = {Lecture Notes in Computer Science},
  volume    = {484},
  pages     = {30--40},
  publisher = {Springer},
  year      = {1991},
  doi       = {10.1007/3-540-53832-1_29}
}

@article{BorodinKostochka1977,
author  = {Borodin, O. V. and Kostochka, A. V.},
  title   = {On an upper bound of a graph's chromatic number, depending on the graph's degree and density},
  journal = {J. Combin. Theory Ser. B},
  volume  = {23},
  number  = {2--3},
  pages   = {247--250},
  year    = {1977},
  doi     = {10.1016/0095-8956(77)90037-5}
}

@article{BosekGrytczukEtAl2019,
author  = {Bosek, B. and Grytczuk, J. and Jak{\'o}bczak, G.},
  title   = {Majority coloring game},
  journal = {Discrete Appl. Math.},
  volume  = {255},
  pages   = {15--20},
  year    = {2019},
  doi     = {10.1016/j.dam.2018.07.020}
}

@article{BruhnDiestelEtAl2010,
author  = {Bruhn, H. and Diestel, R. and Georgakopoulos, A. and Spr{\"u}ssel, P.},
  title   = {Every rayless graph has an unfriendly partition},
  journal = {Combinatorica},
  volume  = {30},
  number  = {5},
  pages   = {521--532},
  year    = {2010},
  doi     = {10.1007/s00493-010-2590-3}
}

@misc{ChawdaNanotiEtAl2026a,
author        = {Chawda, Y. and Nanoti, S. G. and Sankarnarayanan, B.},
  title         = {On the majority game chromatic number of forests and other graphs},
  year          = {2026},
  eprint        = {2609.23803},
  archivePrefix = {arXiv},
  primaryClass  = {math.CO},
  note          = {Preprint}
}

@misc{ChawdaNanotiEtAl2026b,
author        = {Chawda, Y. and Nanoti, S. G. and Sankarnarayanan, B. and Srinivasan, E.},
  title         = {Structural and computational aspects of majority coloring games},
  year          = {2026},
  eprint        = {2609.37796},
  archivePrefix = {arXiv},
  primaryClass  = {math.CO},
  note          = {Preprint}
}

@article{ChenSchelpEtAl1997,
author  = {Chen, G. and Schelp, R. H. and Shreve, W. E.},
  title   = {A new game chromatic number},
  journal = {European J. Combin.},
  volume  = {18},
  number  = {1},
  pages   = {1--9},
  year    = {1997},
  doi     = {10.1006/eujc.1994.0071}
}

@article{CostaPessoaEtAl2020,
author  = {Costa, E. and Pessoa, V. L. and Sampaio, R. and Soares, R.},
  title   = {{PSPACE}-completeness of two graph coloring games},
  journal = {Theoret. Comput. Sci.},
  volume  = {824--825},
  pages   = {36--45},
  year    = {2020},
  doi     = {10.1016/j.tcs.2020.03.022}
}

@unpublished{CowenEmerson1985,
author = {Cowen, R. and Emerson, W.},
  title  = {Proportional colorings of graphs},
  year   = {1985},
  note   = {Unpublished manuscript}
}

@article{Dailey1980,
author  = {Dailey, D. P.},
  title   = {Uniqueness of colorability and colorability of planar 4-regular graphs are {NP}-complete},
  journal = {Discrete Math.},
  volume  = {30},
  number  = {3},
  pages   = {289--293},
  year    = {1980},
  doi     = {10.1016/0012-365X(80)90236-8}
}

@article{FaigleKernEtAl1993,
author  = {Faigle, U. and Kern, W. and Kierstead, H. and Trotter, W. T.},
  title   = {On the game chromatic number of some classes of graphs},
  journal = {Ars Combin.},
  volume  = {35},
  pages   = {143--150},
  year    = {1993},
  url     = {https://combinatorialpress.com/ars-articles/volume-035-ars-articles/on-the-game-chromatic-number-of-some-classes-of-graphs/}
}

@article{FraenkelGoldschmidt1987,
author  = {Fraenkel, A. S. and Goldschmidt, E.},
  title   = {{PSPACE}-hardness of some combinatorial games},
  journal = {J. Combin. Theory Ser. A},
  volume  = {46},
  number  = {1},
  pages   = {21--38},
  year    = {1987},
  doi     = {10.1016/0097-3165(87)90074-4}
}

@article{JanczewskiObszarskiEtAl2022,
author  = {Janczewski, R. and Obszarski, P. and Turowski, K. and Wr{\'o}blewski, B.},
  title   = {Infinite chromatic games},
  journal = {Discrete Appl. Math.},
  volume  = {309},
  pages   = {138--146},
  year    = {2022},
  doi     = {10.1016/j.dam.2021.11.020}
}

@article{JunoszaSzaniawskiRozej2010,
author  = {Junosza-Szaniawski, K. and Ro{\.z}ej, {\L}.},
  title   = {Game chromatic number of graphs with locally bounded number of cycles},
  journal = {Inform. Process. Lett.},
  volume  = {110},
  number  = {17},
  pages   = {757--760},
  year    = {2010},
  doi     = {10.1016/j.ipl.2010.06.004}
}

@article{KalinowskiPilsniakEtAl2025,
author  = {Kalinowski, R. and Pil{\'s}niak, M. and Stawiski, M.},
  title   = {Unfriendly partition conjecture holds for line graphs},
  journal = {Combinatorica},
  volume  = {45},
  number  = {1},
  pages   = {3},
  year    = {2025},
  doi     = {10.1007/s00493-024-00131-1}
}

@incollection{Karp1972,
  author    = {Karp, R. M.},
  title     = {Reducibility among combinatorial problems},
  booktitle = {Complexity of Computer Computations},
  editor    = {Miller, R. E. and Thatcher, J. W. and Bohlinger, J. D.},
  pages     = {85--103},
  publisher = {Plenum Press},
  address   = {New York},
  year      = {1972},
  doi       = {10.1007/978-1-4684-2001-2_9}
}

@article{Koepke2025,
  author  = {Koepke, F.},
  title   = {Solving {Maker--Breaker} games on 5-uniform hypergraphs is {PSPACE}-complete},
  journal = {Electron. J. Combin.},
  volume  = {32},
  number  = {4},
  pages   = {P4.41},
  year    = {2025},
  doi     = {10.37236/13920}
}

@article{KreutzerOumEtAl2017,
author  = {Kreutzer, S. and Oum, S.-I. and Seymour, P. and van der Zypen, D. and Wood, D. R.},
  title   = {Majority colourings of digraphs},
  journal = {Electron. J. Combin.},
  volume  = {24},
  number  = {2},
  pages   = {P2.25},
  year    = {2017},
  doi     = {10.37236/6410}
}

@article{Lawrence1978,
author  = {Lawrence, J.},
  title   = {Covering the vertex set of a graph with subgraphs of smaller degree},
  journal = {Discrete Math.},
  volume  = {21},
  number  = {1},
  pages   = {61--68},
  year    = {1978},
  doi     = {10.1016/0012-365X(78)90147-4}
}

@article{LimaMarcilonEtAl2022,
author  = {Lima, C. V. G. C. and Marcilon, T. and Martins, N. and Sampaio, R.},
  title   = {{PSPACE}-hardness of variants of the graph coloring game},
  journal = {Theoret. Comput. Sci.},
  volume  = {909},
  pages   = {87--96},
  year    = {2022},
  doi     = {10.1016/j.tcs.2022.01.030}
}

@article{Lovasz1966,
author  = {Lov{\'a}sz, L.},
  title   = {On decomposition of graphs},
  journal = {Studia Sci. Math. Hungar.},
  volume  = {1},
  pages   = {237--238},
  year    = {1966},
  url     = {https://real-j.mtak.hu/id/eprint/5451}
}

@incollection{MarcilonMartinsEtAl2020,
author    = {Marcilon, T. and Martins, N. and Sampaio, R.},
  title     = {Hardness of variants of the graph coloring game},
  booktitle = {{LATIN} 2020: Theoretical Informatics},
  series    = {Lecture Notes in Computer Science},
  volume    = {12118},
  pages     = {348--359},
  publisher = {Springer},
  year      = {2020},
  doi       = {10.1007/978-3-030-61792-9_28}
}

@article{PekalaPrzybylo2025,
author  = {P{\k{e}}ka{\l}a, P. and Przyby{\l}o, J.},
  title   = {On list extensions of the majority edge colourings},
  journal = {Electron. J. Combin.},
  volume  = {32},
  number  = {4},
  pages   = {P4.38},
  year    = {2025},
  doi     = {10.37236/13882}
}

@article{RahmanWatson2023,
author  = {Rahman, M. L. and Watson, T.},
  title   = {6-uniform {Maker--Breaker} game is {PSPACE}-complete},
  journal = {Combinatorica},
  volume  = {43},
  number  = {3},
  pages   = {595--612},
  year    = {2023},
  doi     = {10.1007/s00493-023-00026-7}
}

@article{SalesMarcilonEtAl2026,
author  = {Sales, C. L. and Marcilon, T. and Martins, N. and Nisse, N. and Sampaio, R.},
  title   = {The harmonious coloring game},
  journal = {Inform. Process. Lett.},
  volume  = {192},
  pages   = {106609},
  year    = {2026},
  doi     = {10.1016/j.ipl.2025.106609}
}

@article{Schaefer1978,
author  = {Schaefer, T. J.},
  title   = {On the complexity of some two-person perfect-information games},
  journal = {J. Comput. System Sci.},
  volume  = {16},
  number  = {2},
  pages   = {185--225},
  year    = {1978},
  doi     = {10.1016/0022-0000(78)90045-4}
}

@incollection{ShelahMilner1990,
author    = {Shelah, S. and Milner, E. C.},
  title     = {Graphs with no unfriendly partitions},
  booktitle = {A Tribute to Paul Erd{\H{o}}s},
  editor    = {Baker, A. and Bollob{\'a}s, B. and Hajnal, A.},
  pages     = {373--384},
  publisher = {Cambridge University Press},
  year      = {1990},
  doi       = {10.1017/CBO9780511983917.031}
}

@article{Sidorowicz2007,
author  = {Sidorowicz, E.},
  title   = {The game chromatic number and the game colouring number of cactuses},
  journal = {Inform. Process. Lett.},
  volume  = {102},
  number  = {4},
  pages   = {147--151},
  year    = {2007},
  doi     = {10.1016/j.ipl.2006.12.003}
}
}

\newpage
\appendix

\section{Proof of Lemma~\ref{lem:Extended2coloring}}\label{AppendixA}

Suppose $\Phi$ has a satisfying assignment, we show that it is possible to extend the given coloring of $G$ to a majority coloring of $G$.
For each $i\in [1,n]$, if the satisfying assignment sets $x_i$ to True, assign the color red to $x_i$ and the color blue to $\bar{x_i}$.
If the satisfying assignment sets $x_i$ to False, assign the color blue to $x_i$ and the color red to $\bar{x_i}$.
Note that the literal vertices corresponding to true literals (according to the assignment) have been colored red and the literal vertices corresponding to false literals (according to the assignment) have been colored blue.
Now we check that this indeed is a majority coloring of the graph $G$.
As described previously, the vertices in $D,R^\prime$ and $B^\prime$ cannot violate the majority condition.
A vertex $y_i\in Y$ has $k$ blue neighbors in $C$ (when the literal $y_i$ appears in $k$ clauses), $k$ red neighbors in $R^\prime$, as per our description of the coloring.
Therefore, the vertex $y_i$ has less than half its neighbors of its color, and hence satisfies the majority condition.
A vertex $r_i$ ($b_i$) in $R$ ($B$) has two neighbors $x_i$ and $\bar{x_i}$, out of which exactly one is given the color red and the other is given the color blue.
Therefore, the vertices in $R$ ($B$) satisfy the majority condition.
Now consider a vertex $C_j$ in $C$.
Since we have a satisfying assignment, there is at least one literal $y_i$ appearing in clause $C_j$ which is set to true.
Therefore, out of the three neighbors in $Y$ of the clause vertex $C_j$, there is at least one colored red and at most two colored blue.
The clause vertex $C_j$ is colored blue and has two neighbors in $D$, namely $D^1_j$ and $D^2_j$ which are colored red.
Thus, $C_j$ is colored blue and has at least three red neighbors and at most two blue neighbors.
Therefore, the majority condition is satisfied at each vertex in $C$. Thus, we have shown that the majority condition is satisfied at each vertex in $G$.

Now, consider a majority coloring of $G$ which is an extension of the given partial coloring.
We show that there exists a satisfying assignment for the Boolean formula $\Phi$.
First we show that for each $i\in[1,n]$, the vertices $x_i$ and $\bar{x_i}$ are given different colors.
Suppose not; if for some $i$, both $x_i$ and $\bar{x_i}$ are given the color red (blue), then the majority condition is violated at $r_i$ ($b_i$), because $r_i$ ($b_i$) is red (blue) and both its neighbors are red (blue).
So, the vertices $x_i$ and $\bar{x_i}$ get different colors for each $i\in [1,n]$.
Now, if the vertex $x_i$ has the color red, set the variable $x_i$ to True and if the vertex $x_i$ has the color blue, set the variable $x_i$ to False (note that the literal corresponding to the red vertex is True).
Since the coloring of $G$ is a majority coloring, each vertex $C_j$ must have at most two blue neighbors in $Y$, as $C_j$ is colored blue due to the given partial coloring, and it has two red neighbors in $D$ (out of its five neighbors, two in $D$ and three in $Y$). 
Thus, each clause vertex $C_j$ must have at least one red neighbor in $Y$.
Since the literal corresponding to a red vertex in $Y$ is True, the clause $C_j$ is satisfied by the assignment described above.
Hence, extension of the given partial coloring of $G$ to a majority coloring implies the existence of a satisfying assignment of $\Phi$.

\section{Proof that Alice wins the proper coloring game on \texorpdfstring{\(G\)}{G} if and only if she wins the majority precoloring extension game on \texorpdfstring{\(H\)}{H}}\label{AppendixB}

\begin{lemma}\label{lem:fwd}
    If Alice has a winning strategy in the proper coloring game on $G$, then she has a winning strategy in the majority precoloring extension game on $H$.
\end{lemma}

\begin{proof}
If Alice's move on the graph \(G\) in her winning strategy for the proper coloring game is to color the vertex $u \in V(G)$ with color $i$, then her strategy on \(H\) is to color the corresponding vertex \(u^\prime \in V(H)\) with the same color \(i\).
When Bob colors a vertex $v^\prime \in V(H)$, Alice responds by playing according to the strategy in $G$ where Bob would have colored the corresponding vertex $v \in V(G)$ with the same color. Note that Bob can only color a vertex in $V^\prime$ and not in $W$, as the vertices in $W$ are already colored.

We need to show that this is a well-defined strategy.
Suppose not, for the sake of contradiction.
Consider the first instance where the strategy violates the conditions of either the proper coloring game on \(G\) or the majority precoloring extension game on \(H\).
In the former case, it must be that Bob just colored a vertex \(u^\prime \in V(H)\) with the color \(i\), but coloring the vertex \(u \in V(G)\) with the color \(i\) violated the properness condition at \(u\).
Thus there is some neighbor \(v\) of \(u\) which was already colored \(i\), and hence the vertex \(v^\prime \in V(H)\) was already colored \(i\).
So, the vertex \(w^i_{uv} \in V(H)\), which was pre-colored with the color \(i\), would have been majority colored before Bob colored \(u\), and so the color \(i\) would in fact have been forbidden at the vertex \(u\).
In the latter case, it must be that Alice just colored the vertex \(u \in V(G)\) with the color \(i\), but coloring the vertex \(u^\prime \in V(H)\) violated the majority condition at some vertex \(x \in N_H[u^\prime]\).
This vertex \(x\) cannot be \(u^\prime\), since \(u^\prime\) has \(k \geq 2\) neighbors that are pre-colored with \(k\) different colors.
So \(x\) must be of the form \(w^i_{uv}\) for some edge \(uv \in E(G)\).
For the majority condition to be violated at \(w^i_{uv}\) after coloring \(u\) with the color \(i\), it must have been that \(v^\prime\) was already colored with the color \(i\).
But then \(v \in V(G)\) would have been colored \(i\) before the vertex \(u\) was colored, so the color \(i\) would have been forbidden for Alice at the vertex \(u\).
Thus in either case we get a contradiction, and so the strategy is well-defined.

Lastly, we show that this defines a winning strategy for Alice on the majority precoloring extension game on \(H\).
Suppose not, for the sake of contradiction.
By the well-definedness of the strategy, it must be Bob who is forced to make a move which violates the majority condition at a vertex.
That is, for every remaining uncolored vertex \(u^\prime\) of \(H\), and each choice of color \(i\) violates the majority condition at a vertex \(w^i_{uv}\) for some edge \(uv \in E(G)\).
But this means that the other neighbor \(v^\prime\) of \(w^i_{uv}\) must have already been colored \(i\).
In the graph \(G\), this means that at the same stage of partial coloring, for every uncolored vertex \(u\) and every color \(i\), there is some neighbor \(v\) of \(u\) that has been colored \(i\).
This means that Alice has lost the proper coloring game on \(G\), contradicting that she plays as per her winning strategy on \(G\).
Hence, this defines a winning strategy for Alice on the majority precoloring extension game on \(H\).
\end{proof}

The converse of the above lemma also holds. 

\begin{lemma}\label{lem:rev}
    If Alice has a winning strategy in the majority precoloring extension game on $H$, then she has a winning strategy in the proper coloring game on $G$.
\end{lemma}
\begin{proof}
The proof is very similar to the previous one.
If Alice's move on the graph \(H\) in her winning strategy for the majority precoloring extension game is to color the vertex $u^\prime \in V(H)$ with color $i$, then her strategy on \(G\) is to color the corresponding vertex \(u \in V(G)\) with the same color \(i\).
When Bob colors a vertex $v \in V(G)$, Alice responds by playing according to the strategy in $H$ where Bob would have colored the corresponding vertex $v^\prime \in V(H)$ with the same color. Note that both Alice and Bob can only color a vertex in $V^\prime$ and not in $W$, as the vertices in $W$ are already colored.

Again, we need to show that this is a well-defined strategy.
Suppose Alice's move on the graph \(H\) of coloring the vertex \(u^\prime\) with the color \(i\) does not correspond to a proper coloring on the graph \(G\).
Then it must be that some neighbor \(v\) of \(u\) in \(G\) is already colored \(i\).
But then \(v^\prime \in V(H)\) would have already been colored \(i\), and so \(w^i_{uv} \in V(H)\) would have been majority colored with the color \(i\), so the color \(i\) would not have been available on the vertex \(v^\prime\) for Alice in the first place.
Similarly, suppose Bob's move on the graph \(G\) of coloring the vertex \(u\) with the color \(i\) violates the majority condition in \(H\) at a vertex \(w^i_{uv}\) for some edge \(uv\) of \(G\).
Then, it must be that \(v \in V(G)\) was already colored with the color \(i\), so this color was not available at the vertex \(u\) in the first place.
Hence, this strategy is well-defined.

Lastly, we need to show that this is a winning strategy for Alice.
If not, then it must be that it is Bob's turn, and any choice of color on any remaining uncolored vertex in \(G\) violates the properness condition at that vertex.
This means that for each uncolored vertex \(u \in V(H)\) and each color \(i\), there is some neighbor \(w^i_{uv}\) which is majority colored, as the other neighbor of this vertex is already colored with the color \(i\).
But then Alice has lost the majority precoloring extension game on \(H\), contradicting that she plays as per her winning strategy on this graph.
Hence, this defines a winning strategy for Alice on the proper coloring game on \(G\).
\end{proof}

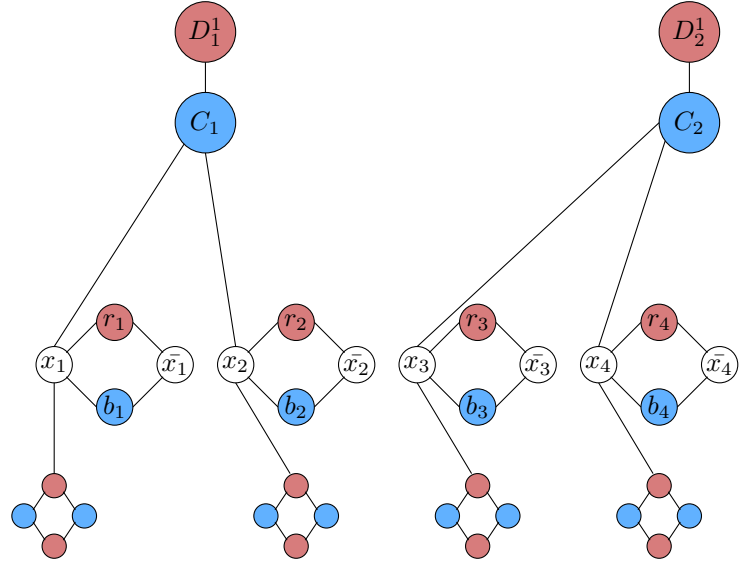
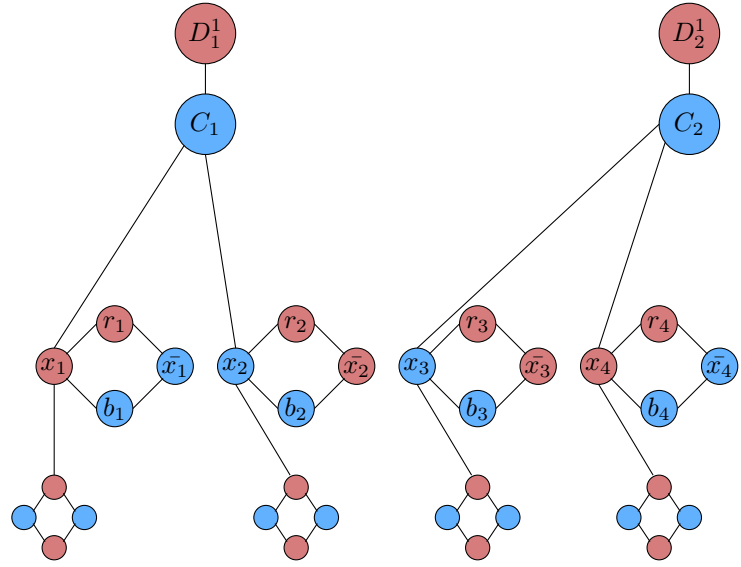
\begin{figure}[!ht]
\begin{subfigure}{\linewidth}
\centering
   \begin{tikzpicture}[scale=0.7, smooth, line width=0.8]
       \foreach \x in {0,3,...,9}
         \draw (\x,0) circle (0.3 cm);
    \foreach \x in {0,1,...,3}
      {\pgfmathtruncatemacro{\z}{\x+1}
      \pgfmathtruncatemacro{\y}{3*\x}
      \node at (\y,0){$x_\z$};}
    \foreach \x in {0,3,...,9}
         \draw (\x+2,0) circle (0.3 cm);
    \foreach \x in {0,1,...,3}
      {\pgfmathtruncatemacro{\z}{\x+1}
      \pgfmathtruncatemacro{\y}{3*\x}
      \node at (\y+2,0){$\bar{x_\z}$};}
    \foreach \x in {0,3,...,9}
        {\draw(\x+0.2,0.2)--(\x+0.7,0.7);
        \draw(\x+0.2,-0.2)--(\x+0.7,-0.7);
        \draw(\x+1.8,0.2)--(\x+1.3,0.7);
        \draw(\x+1.8,-0.2)--(\x+1.3,-0.7);
        }
    \foreach \x in {0,3,...,9}
         {\draw[fill=red!50] (\x+1,0.7) circle (0.3 cm);
         \draw[fill=blue!50](\x+1,-0.7) circle (0.3 cm);
         }
    \foreach \x in {0,1,...,3}
      {\pgfmathtruncatemacro{\z}{\x+1}
      \pgfmathtruncatemacro{\y}{3*\x}
      \node at (\y+1,0.7){$r_\z$};
      \node at (\y+1,-0.7){$b_\z$};}

    \draw[fill=red!50](0,-2) circle (0.2 cm);
    \draw[fill=blue!50](-0.5,-2.5) circle (0.2 cm);
    \draw(-0.35,-2.4)--(-0.15,-2.1);
    
    \draw[fill=red!50](4,-2) circle (0.2 cm);
    \draw[fill=blue!50](3.5,-2.5) circle (0.2 cm);
    \draw(3.65,-2.4)--(3.85,-2.1);

    \draw[fill=red!50](7,-2) circle (0.2 cm);
    \draw[fill=blue!50](6.5,-2.5) circle (0.2 cm);
    \draw(6.65,-2.4)--(6.85,-2.1);
    
     \draw[fill=red!50](10,-2) circle (0.2 cm);
    \draw[fill=blue!50](9.5,-2.5) circle (0.2 cm);     
    \draw(9.65,-2.4)--(9.85,-2.1);
    
    \draw(0,-0.3)--(0,-1.8);
    \draw(3,-0.3)--(3.9,-1.8);
    \draw(6,-0.3)--(6.9,-1.8);
    \draw(9,-0.3)--(9.9,-1.8);
   
    \node (c1) at (1,2) [fill=blue!50, circle, draw, inner sep=1.5pt] {$C_1$};
    \node (c2)at (7,2) [fill=blue!50, circle, draw, inner sep=1.5pt] {$C_2$};
    \node (d11) at (1,3.5) [fill=red!50, circle, draw, inner sep=1.5pt] {$D^1_1$};
    \node (d12) at (7,3.5) [fill=red!50, circle, draw, inner sep=1.5pt] {$D^1_2$};

    \draw(c1)--(d11);
    \draw(c2)--(d12);

    \draw(0,0.3)--(c1);
    \draw(3,0.3)--(c1);
    \draw(6,0.3)--(c2);
    \draw(9,0.3)--(c2);
    
   \end{tikzpicture}
   \caption{A reduced instance from the instance of \problemname{POS-CNF} $\Phi=C_1\land C_2$, where $C_1=x_1\lor x_2$ and $C_2=x_3\lor x_4$. Alice has the following winning strategy: Set $x_1$ to $T$. If Bob does not set $x_3$ to $F$, set $x_3$ to $T$. If Bob sets $x_3$ to $F$, set $x_4$ to $T$.}
\end{subfigure}
\begin{subfigure}{\linewidth}
    \centering
   \begin{tikzpicture}[scale=0.7,
   smooth, line width=0.8]
       \foreach \x in {0,5,8,9}
         \draw[fill=red!50] (\x,0) circle (0.3 cm);
      \foreach \x in {2,3,6,11}
         \draw[fill=blue!50] (\x,0) circle (0.3 cm);
    \foreach \x in {0,1,...,3}
      {\pgfmathtruncatemacro{\z}{\x+1}
      \pgfmathtruncatemacro{\y}{3*\x}
      \node at (\y,0){$x_\z$};}
      \foreach \x in {0,1,...,3}
      {\pgfmathtruncatemacro{\z}{\x+1}
      \pgfmathtruncatemacro{\y}{3*\x}
      \node at (\y+2,0){$\bar{x_\z}$};}
    \foreach \x in {0,3,...,9}
        {\draw(\x+0.2,0.2)--(\x+0.7,0.7);
        \draw(\x+0.2,-0.2)--(\x+0.7,-0.7);
        \draw(\x+1.8,0.2)--(\x+1.3,0.7);
        \draw(\x+1.8,-0.2)--(\x+1.3,-0.7);
        }
    \foreach \x in {0,3,...,9}
         {\draw[fill=red!50] (\x+1,0.7) circle (0.3 cm);
         \draw[fill=blue!50](\x+1,-0.7) circle (0.3 cm);
         }
    \foreach \x in {0,1,...,3}
      {\pgfmathtruncatemacro{\z}{\x+1}
      \pgfmathtruncatemacro{\y}{3*\x}
      \node at (\y+1,0.7){$r_\z$};
      \node at (\y+1,-0.7){$b_\z$};}

    \draw[fill=red!50](0,-2) circle (0.2 cm);
    \draw[fill=blue!50](-0.5,-2.5) circle (0.2 cm);
    \draw(-0.35,-2.4)--(-0.15,-2.1);
    
    \draw[fill=red!50](4,-2) circle (0.2 cm);
    \draw[fill=blue!50](3.5,-2.5) circle (0.2 cm);
    \draw(3.65,-2.4)--(3.85,-2.1);

    \draw[fill=red!50](7,-2) circle (0.2 cm);
    \draw[fill=blue!50](6.5,-2.5) circle (0.2 cm);
    \draw(6.65,-2.4)--(6.85,-2.1);
    
     \draw[fill=red!50](10,-2) circle (0.2 cm);
    \draw[fill=blue!50](9.5,-2.5) circle (0.2 cm);     
    \draw(9.65,-2.4)--(9.85,-2.1);
    
    \draw(0,-0.3)--(0,-1.8);
    \draw(3,-0.3)--(3.9,-1.8);
    \draw(6,-0.3)--(6.9,-1.8);
    \draw(9,-0.3)--(9.9,-1.8);
   
    \node (c1) at (1,2) [fill=blue!50, circle, draw, inner sep=1.5pt] {$C_1$};
    \node (c2)at (7,2) [fill=blue!50, circle, draw, inner sep=1.5pt] {$C_2$};
    \node (d11) at (1,3.5) [fill=red!50, circle, draw, inner sep=1.5pt] {$D^1_1$};
    \node (d12) at (7,3.5) [fill=red!50, circle, draw, inner sep=1.5pt] {$D^1_2$};

    \draw(c1)--(d11);
    \draw(c2)--(d12);

    \draw(0,0.3)--(c1);
    \draw(3,0.3)--(c1);
    \draw(6,0.3)--(c2);
    \draw(9,0.3)--(c2);
     
\end{tikzpicture}
    \caption{Example of a game-play where Alice plays by the above strategy: Alice colors $x_1$ red, Bob colors $\bar{x_3}$ red, Alice colors $x_4$ red. Then Bob colors $x_2$ blue, Alice colors $\bar{x_1}$ blue, Bob colors $\bar{x_2}$ red, Alice colors $x_3$ blue, Bob colors $x_4$ blue. $G$ is now majority colored.}
\end{subfigure}
    \caption{Construction of a reduced instance of \problemname{Majority \(2\)-precoloring Extension Game} from a given instance of \problemname{POS-CNF} and demonstration of a game play when Alice has a winning strategy.}
    \label{fig:reductioncnfAlice}
\end{figure}

\section{Proofs of Lemmas~\ref{lem:Alicewin} and~\ref{lem:Bobwin}}\label{AppendixC}

\begin{proof}[Proof of Lemma~\ref{lem:Alicewin}]

Note that all the vertices of the graph $G$ end up being colored when Alice plays according to this strategy and we now show that the coloring obtained is indeed a majority coloring of the graph $G$, and hence all the moves of the players were valid in the majority $2$-precoloring extension game.
Clearly, all the vertices in $D$ satisfy the majority condition because all the vertices in $D$ are red and have exactly one neighbor (in $C$), which is blue.
The vertices in $R^\prime$ satisfy the majority condition because they are colored red and each has at most one red neighbor and exactly one blue neighbor.
The vertices in $B^\prime$ satisfy the majority condition because they are colored blue and each has exactly one red neighbor.
Therefore, the majority condition is satisfied by the vertices in $D\cup B^\prime \cup R^\prime$. 

Note that for each $i\in[1,n]$, the vertices $x_i$ and $\bar{x_i}$ cannot have the same color.
Clearly, Alice never assigns the same color to $x_i$ and $\bar{x_i}$ as per the given strategy.
Suppose $x_i$ ($\bar{x_i}$) is colored red (blue) and Bob attempts to color $\bar{x_i}$ ($x_i$) red (blue): this will not be possible according to the rules of the game as the majority condition will be violated at $r_i$ ($b_i$).
Therefore, $x_i$ and $\bar{x_i}$ must have opposite colors.

Since $x_i$ and $\bar{x_i}$ must have opposite colors (for each $i\in [1,n]$), the majority condition is satisfied at each $r_i$ and each $b_i$ (for all $i\in [1,n]$).
Each negative vertex $\bar{x_i}$ has two neighbors $r_i$ and $b_i$, out of which one is red and the other is blue, hence the majority condition is satisfied at $\bar{x_i}$, irrespective of whether $\bar{x_i}$ is colored red or blue.
Therefore, the majority condition is satisfied by the vertices in $R\cup B\cup L$. 

Consider a positive vertex $x_i$ (where $i\in [1,n]$) in $X$, such that the variable $x_i$ appears in $k$ clauses of $\Phi$.
By the construction, the positive vertex $x_i$ has $k$ blue neighbors in $C$, $k$ red neighbors in $R$, one red neighbor $r_i$, and one blue neighbor $b_i$.
Therefore, the majority condition is not violated at $x_i$.

Now consider a clause vertex $C_j$ (where $j\in [1,m]$) in $C$. Suppose the clause $C_j$ has $k$ literals in $\Phi$.
Then the clause vertex $C_j$ has $k-1$ red neighbors in $D$ and $k$ neighbors in $X$. Since Alice wins in the \problemname{POS-CNF} game, in the final assignment there exists a variable $x_i$ in $C_j$ which was set to True by Alice (recall that we have assumed without loss of generality that Bob never sets any variables to True).
According to Alice's strategy in the majority $2$-precoloring extension game, she must have colored the positive vertex $x_i$ red.
This is possible, because if Bob had already colored $\bar{x_i}$ red, then according to Alice's strategy, she considers the variable $x_i$ to be already set to False, hence the strategy would not have permitted her to color $x_i$.
Coloring the vertex $x_i$ red also cannot violate the majority condition at any clause vertex because all the clause vertices are colored blue.
When $x_i$ is colored by Alice, $\bar{x_i}$ was uncolored as seen above, and thus the majority condition is not violated at $r_i$ or $b_i$.
Lastly, as we have seen above, coloring a vertex in $X$ does not violate the majority condition at a vertex in $R$.
Thus, Alice has assigned the color red to $x_i$ and hence the clause vertex $C_j$ has at least one red neighbor and at most $k-1$ blue neighbors in $X$.
Thus, the majority condition is satisfied at each vertex $C_j$ in $C$.
\end{proof}

\begin{proof}[Proof of Lemma~\ref{lem:Bobwin}]

Consider the final assignment of the variables after Bob plays according to the winning strategy of the \problemname{POS-CNF} game.
Since this is a winning strategy for Bob, there exists a clause $C_j=x_{i_1}\lor x_{i_2} \lor \dotsb \lor x_{i_k}$ such that $C_j$ is False, i.e., $x_{i_1}, x_{i_2}, \dotsc, x_{i_k}$ are all set to False.
Now consider the situation in the majority $2$-precoloring extension game, when Bob has set $x_{i_1}, x_{i_2}, \dotsc, x_{i_k}$ to False in the POS-CNF game (recall that Alice never sets a variable to False).
The negative vertices $\bar{x_{i_1}}, \bar{x_{i_2}}, \dotsc, \bar{x_{i_k}}$ are colored red by Bob (because if Alice had colored either $x_{i_\ell}$ or $\bar{x_{i_\ell}}$, then in the strategy for Bob, the variable $x_{i_\ell}$ is assumed to be set to True).
Now, all the vertices $x_{i_1}, x_{i_2}, \dotsc, x_{i_k}$ must be colored blue because coloring any $x_{i_\ell}$ red violates the majority condition at $r_{i_\ell}$.
Without loss of generality, we assume that $x_{i_1}, x_{i_2}, \dotsc, x_{i_{k-1}}$ are colored blue when $x_{i_k}$ is uncolored.
Recall that $x_{i_k}$ cannot be colored red because $\bar{x_{i_k}}$ is colored red and the majority condition at $r_{i_k}$ will be violated.
If $x_{i_k}$ is colored blue, then the clause vertex $C_j$ has $k$ blue neighbors ($x_{i_1}, x_{i_2}, \dotsc, x_{i_k}$), and $k-1$ red neighbors $D^1_j, D^2_j, \dotsc, D^{k-1}_j$.
Therefore, the majority condition is violated at $C_j$, since $C_j$ is colored blue.
Therefore, the positive vertex $x_{i_k}$ cannot be assigned the color blue as well.
Hence, all the vertices of the graph $G$ cannot be colored and Bob wins in the majority $2$-precoloring extension game.
\end{proof}

\end{document}